\documentclass[onenumber, plainnumbering,english,11pt]{amsart}
\usepackage[applemac]{inputenc}
\usepackage[square,numbers]{natbib}
\usepackage{amssymb,amsmath,amscd,mathrsfs,paralist}

\usepackage[mathscr]{eucal}
\usepackage{bm}
\usepackage[T1]{fontenc}
\usepackage{graphicx}
\usepackage{tikz}
\usepackage{cool}
\usepackage[all,cmtip]{xy}
\usepackage[colorlinks,linkcolor=black,urlcolor=black]{hyperref}
\hypersetup{
     colorlinks   = true,
     citecolor    = black}
\usepackage{pgfplots}
\newcommand{\DistTo}{\xrightarrow{
   \,\smash{\raisebox{-0.65ex}{\ensuremath{\scriptstyle\sim}}}\,}}
\newcommand{\R}{\mathbb R}
\newcommand{\N}{\mathbb N}
\newcommand{\Z}{\mathbb Z}
\newcommand{\C}{\mathbb C}

\newcommand{\Q}{\mathbb Q}

\newcommand{\PGa}{PGL_2(\mathbb C)}

\newcommand{\lesss}{\rotatebox[origin=c]{90}{$\land$}}
\newcommand{\less}{\ \lesss\ }
\newcommand{\biggg}{\rotatebox[origin=c]{90}{$\lor$}}
\newcommand{\bg}{\ \biggg\ }
\newcommand{\nc}{\newcommand}
\nc{\BCc}{{\mathbb{C}(\wp(z),\wp^\prime(z))}}
\nc{\BC}{{\mathbb C}}
\nc{\BQ}{{\mathbb Q}}
\nc{\BR}{{\mathbb R}}
\nc{\BZ}{{\mathbb Z}}
\nc{\BP}{{\mathbb P}}
\nc{\BN}{{\mathbb N}}
\nc{\BM}{{\mathbb M}}
\nc{\fH}{{\mathfrak{H}}}
\nc{\vp}{{\varepsilon}}\nc{\dpar}{{\partial}}\nc{\al}{{\alpha}}
\nc{\PSL}{PSL(2,\BR)}
\nc{\PS}{PSL(2,\BZ)}
 \nc{\CL}{PSL(2,\BZ/m\BZ)}
 \newtheorem{theorem}{Theorem}[section]
\newtheorem{corollary}[theorem]{Corollary}
\newtheorem{lemma}[theorem]{Lemma}
\newtheorem{proposition}[theorem]{Proposition}
\theoremstyle{definition}
\newtheorem{definition}[theorem]{Definition}
\newtheorem{remark}[theorem]{Remark}
\newtheorem{example}[theorem]{Example}

\newcommand{\theoref}[1]{Theorem~\ref{#1}}

\newcommand{\lemref}[1]{Lemma~\ref{#1}}

\DeclareMathOperator{\Id}{Id}

\numberwithin{equation}{section}  
\begin{document}
\title{Frobenius determinants and Bessel Functions}
\author{Ahmed Sebbar and Oumar Wone}
\address{Ahmed Sebbar and Oumar Wone\\
Univ. Bordeaux, IMB 
   \\UMR 5251, F-33400 Talence, France.}
    \email{ahmed.sebbar@math.u-bordeaux.fr}
      \email{oumar.wone@math.u-bordeaux.fr}
\date{}
\begin{abstract}
We study the geometry and partial differential equations arising from the consideration of Frobenius determinants, also called-group-determinants. This leads us to address some aspects of twistor theory as well as some extensions of Bessel functions.
\end{abstract}
\maketitle 
\tableofcontents

\section{Introduction}
Let $G$ be a finite group, and $A$ be a commutative  $\mathbb K$-algebra, where $\mathbb K$ is a field. We fix some elements $(x_{g\in G})$ in $A$, indexed by $G$.
Dedekind and Frobenius considered  the group  determinant
\[\Theta_G(x)= {\det}(x_{gh^{-1}}),\quad x= (x_{g\in G})\]
This is a homogeneous polynomial in the $x_g$'s of degree the order of $G$, with  integer coefficients. We refer to \cite{conrad} for a nice exposition of the history of this determinant. We will be mainly concerned with cyclic group $G$, that we assume to be $\displaystyle \BZ/n\BZ   $ which amounts to circulant matrices and circulant determinants. Two fundamental group determinants are of interest for us in this work: the case of $A= {\mathbb K}[X_1,\cdots, X_n]$  the algebra of polynomials in $ X_1,\cdots, X_n$ with coefficients in $ \mathbb K $ and the algebra
$A= {\mathbb K}[\partial_1,\cdots, \partial_n]$ where $\partial_i= \dfrac{\partial}{\partial X_i }  $. The case of $n=2$ is particularly simple and illuminating. This corresponds to the cases of circle and the Euclidean Laplacian:

\begin{equation}\label{circle}
\begin{vmatrix} x & iy \\  \\ iy & x \end{vmatrix}= x^2+y^2,\quad \begin{vmatrix} \dfrac{\partial}{\partial x} & i\dfrac{\partial}{\partial x} \\ \\ i\dfrac{\partial}{\partial x} &\dfrac{\partial}{\partial x}  \end{vmatrix}= \dfrac{\partial^2}{\partial x^2}+\dfrac{\partial^2}{\partial y^2}
\end{equation}

and the case of the hyperbola and the wave operator:

\begin{equation}\label{hyperbola}
\begin{vmatrix} x & y \\  \\ y & x \end{vmatrix}= x^2-y^2,\quad \begin{vmatrix} \dfrac{\partial}{\partial x} & \dfrac{\partial}{\partial x} \\ \\ \dfrac{\partial}{\partial x} &\dfrac{\partial}{\partial x}  \end{vmatrix}= \dfrac{\partial^2}{\partial x^2}-\dfrac{\partial^2}{\partial y^2}.
\end{equation}

The group $\displaystyle \BZ/n\BZ   $  is an abelian group so its irreducible representations are all of degree one. If $\displaystyle \omega= \omega_n= e^{\frac{2i\pi}{n}}    $, the 
$n$ characters are given by
\[\chi_r(s)= \omega_n^{rs},\; s\in \BZ/n\BZ , \;r\in \{0,1,\cdots, n-1\}\]
so that

\begin{equation}\label{determinant}
\Theta_{\BZ/n\BZ}(X)=\prod_{j=0}^{n-1}\left(\sum_{k=0}^{n-1}\omega ^{jk}   X_k\right)= \prod_{j=0}^{n-1}\left(X_0+\omega^jX_1+\cdots+\omega^{(n-1)j}X_{n-1} \right)
\end{equation}

This can be seen as generalized  Brahmagupta's identity, for when $G= \BZ/2\BZ$, we obtain the group structure on the hyperbola $x^2-dy^2=1, d>0   $ or,
what is the same
\[(x_1^2-dy_1^2)(x_2^2-dy_2^2)=(x_1x_2+dy_1y_2)^2-d( x_1y_2+x_2y_1)^2\]

The equation \eqref{determinant}, when the variables are partial derivations $\displaystyle \dfrac{\partial}{\partial x_k},\,1\leq k\leq n   $, gives rise to a differential operator 
\[{\mathcal D}= \prod_{j=1}^{n}\left(\sum_{k=1}^{n}\omega ^{jk} \dfrac{\partial}{\partial x_i} \right)= \prod_{j=1}^{n}\left(\dfrac{\partial}{\partial x_1}+\omega^j\dfrac{\partial}{\partial x_2}+\cdots+\omega^{(n-1)j}\dfrac{\partial}{\partial x_n} \right) ,\]
 which by the change of variables 
 \[u_j=  x_1+\omega^j x_2+\cdots+\omega^{(n-1)j} x_n\]
 becomes 
 \[D= \prod_{j=1}^{n} \dfrac{\partial}{\partial u_j}.\]
 which is a separable differential operator.
We restrict ourselves in this work to the case $n = 2$ and then $G= \BZ/2\BZ$, and more particularly on harmonic functions. This case is already rich enough to make a link with twistors theory, and Bessel functions appear. These features have many facets, worthy of interest, and we explore a few. This gives us the opportunity to show the influence of group theory in the context of classical Bessel functions, emphasizing the operator side and also some aspects of differential Galois theory. The group aspect is developed in \cite{vilenkin}, \cite{miller}, \cite{talman}. One can also consider harmonic functions in the three dimensional case, even it is not covered by Frobenius determinant, which is
\[\frac{\partial^3}{\partial x^3}+\frac{\partial^3}{\partial y^3}+\frac{\partial^3}{\partial z^3}-3\frac{\partial}{\partial x}\frac{\partial}{\partial y}\frac{\partial}{\partial z},\]
which differs from the Euclidean three dimensional Laplacian.

To not be too long, we hope to return on it on another occasion for a more general study. We recall that Frobenius's determinant, in its simplest form, ie $ n = 2 $ gives the Laplacian or wave operator in two dimensions. To emphasize the link with harmonic functions, we discuss the most general representation of harmonic functions, initially given by Whittaker and Bateman, in the form of John transforms. We dedicate enough time to present, in a self contained manner, John transform of certain distributions associated with $\bf x _ + ^ a $ which are then given by hypergeometric functions, and thus, are multivalued functions on ${\mathbb P}^1\setminus \{0,1,\infty\}$. The uniformization and  modular function $\lambda$ allow us to present these John transforms as uniform functions on the upper half-plane $\mathcal H$.

The origin of integral transforms can be traced back to Radon in 1917. He studied transforms of functions $f:\R^2\to\R$ with decay conditions at infinity. More precisely to $L\subset \R^2$, an oriented line, one associates
 $$\displaystyle\phi(L):=\int_Lf.$$
 Remark that there exists an inversion formula for the Radon transform, see \cite{gelfand2000, radon1917}. The next important study of integral transforms is due to F. John in $1938$ \cite{gelfand2000, john1938}. He was interested in integral transforms of functions $f:\R^3\to\R$ (satisfying appropriate conditions which make the integrals well-defined). He defined, for $L$ an oriented line, the transform
 $$\displaystyle\phi(L):=\int_Lf$$
 or again
 $$\displaystyle\phi(\alpha_1,\alpha_2,\beta_1,\beta_2)=\int_{-\infty}^{+\infty}f(\alpha_1s+\beta_1,\alpha_2s+\beta_2,s)ds.$$
 The space of oriented lines in $\R^3$ is four dimensional and $3<4$ so one excepts a condition on $\phi$. Indeed if we differentiate under the integral sign we get the ultrahyperbolic wave equation
 $$\displaystyle\dfrac{\partial^2\phi}{\partial\alpha_1\partial\beta_2}=\dfrac{\partial^2\phi}{\partial\alpha_2\partial\beta_1}.$$
 Changing coordinates $\alpha_1=x+y,\,\alpha_2=t+z,\,\beta_1=t-z,\,\beta_2=x-y$, the ultrahyperbolic wave equation becomes
 $$\displaystyle\dfrac{\partial^2\phi}{\partial x^2}+\dfrac{\partial^2\phi}{\partial z^2}-\dfrac{\partial^2\phi}{\partial y^2}-\dfrac{\partial^2\phi}{\partial t^2}=0.$$ 
 Finally in 1967, R. Penrose \cite{penrose1967} introduced the twistor transform in order to find solutions of the wave equation in Minkowski space. More precisely the solutions of the equation 
 $$\dfrac{\partial^2\phi}{\partial t^2}=\dfrac{\partial^2\phi}{\partial x^2}+\dfrac{\partial^2\phi}{\partial y^2}+\dfrac{\partial^2\phi}{\partial z^2}$$
 are expressed in the form
 $$\displaystyle\phi(x,y,z,t)=\int_{\Gamma\subset\mathbb{P}^1(\C)}f((z+t)+(x+iy)\lambda,(x-iy)-(z-t)\lambda,\lambda)d\lambda.$$
Here $\Gamma\subset\mathbb{P}^1(\C)$ is a closed contour and the function $f$ is holomorphic on $\mathbb{P}^1(\C)$ except for some number of poles. This marked the introduction of sophisticated methods such as cohomology in the study of such questions.

\section{The John transform and Uniformization}
\subsection{The Cross-ratio and the modular function $\lambda$ }We show in this section that the John transform of the distribution $ {\bf x}^{a}_+$ is a hypergeometric function, which we know has an analytic continuation, as a multivalued function,  to the universal covering $\widetilde{{\mathbb P}^1\setminus \{0,1,\infty\}}$, $ {\mathbb P}^1= {\mathbb P}^1(\C)$. We recall that the analytic continuation of the hypergeometric function $F(a,b,c;x)$ is multivalued on $\mathbb{P}^1\setminus\{0,1,\infty\}$. Now in order to show the richness of the John transform we show that the transform of $ {\bf x}^+$ (see below) has very interesting modular properties, due to its link with the modular function $\lambda$. This is in essence a representation of hypergeometric function in terms of the even Jacobi theta zero-value(Nullwerte). We give a self-contained presentation of this idea.

 There is a deep link between cross-ratios and Uniformization \cite{donaldson}. Let $a_1, a_2, a_3, a_4 \in (\mathbb {P}^1\times\mathbb {P}^1\times\mathbb {P}^1\times\mathbb {P}^1)$   be four different points on the projective line. 
 The cross-ratio $[a_1, a_2, a_3,a_4]$ is defined as
\[\lambda= [a_1, a_2, a_3,a_4]:=  \frac{(a_1-a_3)(a_2-a_4)}{ (a_1-a_4)(a_2-a_3)}.\]
If one of the $a_i$ is $\infty$,  the factors containing it cancel each other out, for example if $a_4=\infty$
\[[a_1, a_2, a_3,a_4]:= \frac{a_1-a_3}{a_2-a_3}.\]
 The cross-ratio is $\PGa$-invariant. Let $\Delta$ be the set of 4-tuples of distinct points in $\BP^1$ with at least two points equal, and 
 \[X={(\BP^1)}^4\setminus \Delta.\]
 The cross-ratio defines a bijective map
 \[R: X/\PGa   \longrightarrow  \BC\setminus \{0,1\}.\]
 Let $ \Gamma(2)$ be the congruence subgroup of the modular group
 \[\Gamma(2)= \left\{\begin{pmatrix}a&b\\c&d  \end{pmatrix} \in {\rm SL}_2(\BZ): a,d\equiv 1({\rm mod}\, n), \,b,c \equiv 0({\rm mod}\, n)\right\}.\]
 It acts on the upper half-plane $\displaystyle {\mathcal H}= \left\{z\in \BC, \Im z>0   \right\} $ by projective transformation
 \[\begin{pmatrix}a&b\\c&d  \end{pmatrix} \in \Gamma(2),\quad  z\in \BC \longrightarrow \frac{az+b}{cz+d}.\] 
 With this action we have a natural bijection between the set of ordered 4-tuples of distinct points in $\BP^1  $ modulo projective transformations and the points on
  $\displaystyle  {\mathcal H}/\Gamma(2) $ \cite{dolgachev} (p.94).\\
  
 Cross-ratios have the following obvious, but important properties: they are invariant under Moebius transformations
and under the action of the Klein four group $V_4 < A_4$ (alternating group) acting via permutation of the components with
\[V_4= \{{\rm Id}, (12)(34), (13)(24), (14)(23)\}.\]
The Klein group $V_4$  is a normal subgroup of $S_4$, the group of permutations of four elements and $V_4 \cap S_3= {\rm Id}$. The group $S_3$ can be identified
with the subgroup of $S_4$ of permutations of four letters fixing one fixed letter. Hence $V_4 S_3= S_4.   $. Each permutation in $S_4$ 
\[\sigma= \begin{pmatrix}
1&2&3&4\\
\alpha_1&\alpha_2&\alpha_3&\alpha_4\\
\end{pmatrix}\]
induces a map on cross-ratios:
\[\lambda= [a_1, a_2, a_3,a_4] \longrightarrow \lambda_{\sigma}= [a_{\sigma(1)}, a_{\sigma(2)}, a_{\sigma(3)},a_{\sigma(4)}]\] 
We thus obtain a group $G$ of fractional linear transformations and the map of $S_4$ onto $G$ is a morphism with the Klein four group $V_4$ as kernel. Hence the action of $S_4$ on $\BP^1\setminus \{0,1,\infty\} $ reduces to the action of $S_3=\left\{\tau, \sigma, \tau^2= \sigma^3=(\tau \sigma)^2=1 \right\}$ on the same space. So given $t\in \BP^1\setminus \{0,1,\infty\} $, the different six images are
\[{\rm Id}: t \to t;\;\tau: t \to \frac{1}{t};\;\sigma: t \to \frac{1}{1-t};\;\sigma^2: t \to1- \frac{1}{t};\;\tau\sigma: t \to 1-t;\;\tau \sigma^2: t \to \frac{t}{t-1}.\]

\noindent Furthermore any cubic polynomial can be reduced to the Weierstrass normal form
  \[p(z)= 4z^3-c_2z-c_3, \quad c_1,c_2\in \BC\]
 by means of a transformation $ z\mapsto az+b,\; a,b\in \BC, a\neq 0$. If  $e_1, e_2, e_3$  are the roots of $p(z)$, the discriminant is defined by
 \[\Delta_p= 16(e_1-e_2)^2(e_2-e_3)^2(e_3-e_1)^2\]
 which is not equal to zero if and only of these roots are distinct. Between the roots and the coefficients of the cubic polynomial we have
 \begin{align*}
 e_1+e_2+e_3= &0\\
  e_1e_2+e_2e_3+e_3e_1= & -\frac{c_2}{4}\\
  e_1e_2e_3= &\frac{c_3}{4}
 \end{align*}
Moreover  the discriminant is given by
\[\Delta_p=c_2^3-27c_3^2.\]
A classical and important result in elliptic functions theory \cite[p.~287, cor.~6.5.8]{singerman} is that if $c_2, c_3,\in \BC, c_2^3-27c_3^2\neq 0$ are given, then there is a lattice 
\[\Omega= \left\{2m\omega_1+2n\omega_3,\;m,n\in \BZ   \right\}\]
with $\displaystyle \omega_1\neq 0,\;  \omega_3\neq 0,\;   \frac{\omega_1}{\omega_3}\notin \BR   $ such that
\[ g_2= g_2(\Omega)= 60 \sum\nolimits'_{\omega\in \Omega} \omega^{-4}\]
and
\[g_2= g_2(\Omega)= 140 \sum\nolimits'_{\omega\in \Omega} \omega^{-6}.\]
We define 
\[\Delta(\Omega)= g_2^3-27g_3^2\]
and the Weierstrass $\wp$-function
\[\wp(u)= \frac{1}{u^2}+ \sum\nolimits'_{\omega\in \Omega}\left(\frac{1}{(u-\omega)^2}-\frac{1}{\omega^2}\right).\]
This function satisfies the nonlinear differential equation
\[\wp'{^2}(u)= 4\wp^3(u)-g_2\wp(u)-g_3\]
This means that if $z= \wp(u) $, then
\[\left(\frac{d z}{d u}\right)^2= 4z^3-g_2z-g_3\]
so that
\[\wp^{-1}(z)= u= \int \frac{dz}{\sqrt{p(z)}}.\]
The function $\wp$ is an even elliptic function satisfying to
\[ \wp'(\omega_1)=  \wp'(\omega_2)=  \wp'(\omega_3)=0\]
and for $j=1,2,3$:
\[e_j= \wp(\omega_j)= \omega_j^{-2}+\sum\nolimits'_{\omega\in \Omega}\left(\frac{1}{(\omega_j-\omega)^2}-\frac{1}{\omega^2}\right)\]

The periods $2\omega_1, 2\omega_3$ and $2\omega_3= 2\omega_1+ 2\omega_3$ can be recovered by
\\
\begin{align*}
\omega_1=&\frac{1}{2}\int_{e_1}^{\infty} \frac{dt}{\sqrt{(t-e_1)(t-e_2)(t-e_3)}}=& \frac{1}{2}\int_{e_3}^{e_2}\frac{dt}{\sqrt{(t-e_1)(t-e_2)(t-e_3)}}\\
\\
\omega_2=&\frac{1}{2}\int_{e_2}^{\infty}\frac{dt}{\sqrt{(t-e_1)(t-e_2)(t-e_3)}}=&\frac{1}{2} \int_{e_3}^{e_2}\frac{dt}{\sqrt{(t-e_1)(t-e_2)(t-e_3)}}\\
\\
\omega_3=&\frac{1}{2}\int_{e_3}^{\infty}\frac{dt}{\sqrt{(t-e_1)(t-e_2)(t-e_3)}}=&\frac{1}{2} \int_{e_3}^{e_2}\frac{dt}{\sqrt{(t-e_1)(t-e_2)(t-e_3)}}.
\end{align*} 
\\
These integrals are hypergeometric functions in disguise. They can be brought to the Riemann normal form by the change of variables:
\[u= e_3+\frac{e_1-e_3}{z^2}\]
which gives
\[K= \sqrt{e_1-e_3} \omega_1,\quad iK'= \sqrt{e_1-e_3} \omega_3,\]
where
\begin{equation}\label{periods!}
K= \int_0^1 \frac{dz}{\sqrt{(1-z^2)(1-k^2z^2)}},\quad K= \int_0^1 \frac{dz}{\sqrt{(1-z^2)(1-k'{^2}z^2)}},
\end{equation}
and
\[ k^2=\frac{e_2-e_3}{e_1-e_3},\quad k'{^2}=\frac{e_2-e_1}{e_3-e_1}.\]
The classical Gauss hypergeometric function is defined by the power series for $\vert z\vert<1$:
\begin{equation}\label{hypergeometric1}
\setlength\arraycolsep{1pt}
{}_2 F_1\left(a,b;c;z\right)= \sum_{n=0}^{\infty} \frac{(a)_n}{(b)_n} \frac{z^n}{n!},\quad (a)_0=1,\;\; (a)_n= a(a+1)\cdots(a+n-1). 
\end{equation}
which admits the integral representation
\begin{equation}\label{hypergeometric2}
\setlength\arraycolsep{1pt}
{}_2 F_1\left(a,b;c;z\right)= \frac{\Gamma(c)}{\Gamma(b)\Gamma(c-b)}\int_0^1t^{b-1}(1-t)^{c-b-1}(1-zt)^{-a}\,dt.
\end{equation}
This function satisfies the following differential equation of the second order
\begin{equation}\label{hypergeo}
z(1-z)\frac{d^2y}{dz^2}+\left(c-(a+b+1)z\right)\frac{dy}{dz}-ab y=0
\end{equation}
An crucial point \cite[$\S$.~22.301, p.~499]{WW} and \cite[p.~268]{donaldson} is that the periods $K,\;K'$ are hypergeometric functions with respect to the variable $k^2$ and verify a hypergeometric differential equation, which fits into a general theory of Picard-Fuchs equations and Gauss-Manin connections. Indeed
\[K= \frac{\pi}{2}\; {}_2 F_1\left( \frac{1}{2}, \frac{1}{2};1,k^2\right),\quad K'=\frac{\pi}{2}\; {}_2 F_1\left( \frac{1}{2}, \frac{1}{2};1,k'{^2}\right)\]
and the general solution of the differential equation
\[z(1-z)\frac{d^2y}{dz^2}+\left(1-2z\right)\frac{dy}{dz}-\frac{1}{4}y=0\]
has the form $\alpha K(z)+ \beta K'(z),\;z= k^2  $
This has the following consequence: the cross-ratio
\[ k^2=\frac{e_2-e_3}{e_1-e_3}\]
is, by homogeneity, a function of $\tau$ alone. We may assume $\Im \tau>0  $. we obtain a modular function $\lambda(\tau)= k^2= k^2(\tau)$ with respect to congruence subgroup $\Gamma(2)$ of the modular group ${\rm SL}_2(\BZ) $. It can be shown that actually we have an isomorphism:
\[ \lambda: {\mathcal H}/\Gamma(2) \DistTo   \BC\setminus\{0,1\}\]
and the the inverse function may be expressed as a quotient of hypergeometric functions
\begin{equation}\label{hypergeo}
z= i \frac{{}_2 F_1\left( \frac{1}{2}, \frac{1}{2};1,k'{^2}\right)}{{}_2 F_1\left( \frac{1}{2}, \frac{1}{2};1,k^2\right)}.
\end{equation}
The interplay between the geometry of the cross-ratio, the structure of the groups $ \Gamma (2) $ and that of the permutations $ S_4 $ of four elements leads to a fascinating and elegant formula for the hypergeometric function. We start with the identity \eqref{hypergeometric2}
\begin{equation*}
\frac{\Gamma(b)\Gamma(c-b)}{\Gamma(c)}{}_2 F_1\left(a,b;c;z\right)= \int_0^1t^{b-1}(1-t)^{c-b-1}(1-zt)^{-a}\,dt= \int_0^1 U(t)\,dt\\
\end{equation*}
If 
\[t=[vd_3d_1d_2],\quad z=[a_4a_3a_2a_1]\]
then the desired formula is:
\begin{align}\label{hypergeometric3}
\frac{\Gamma(b)\Gamma(c-b)}{\Gamma(c)}{}_2 F_1&\left(a,b;c;[a_4a_3a_2a_1]\right)
\\
\nonumber= \int_0^1&[va_3a_1a_2]^{b-1}(va_1a_3a_2)^{c-b-1}(va_1a_4a_2)^{-a}\,d[va_3a_1a_2].
\end{align}
\begin{proposition}
If $v=\wp(u)= \wp(u, 2\omega_1, 2\omega_3),\;d_1=\infty,\;d_2=e_3,\;d_3=e_1,\;d_4=e_2$ with $e_1+e_2+e_3=0$ then the equality \eqref{hypergeometric3} becomes
\begin{align}
\frac{1}{2}\frac{\Gamma(b)\Gamma(c-b)}{\Gamma(c)}{}_2 F_1&\left(a,b;c;\frac{e_2-e_3}{e_1-e_3}\right)\\
\nonumber&= \int_0^{\omega_1}\left(\wp(u)-e_1\right)^{c-b-\frac{1}{2}} \left( \wp(u)-e_2\right)^{\frac{1}{2}-\alpha} \left( \wp(u)-e_3\right)^{\alpha-\gamma+\frac{1}{2}}\;du
\end{align}
\end{proposition}
\subsection{The distribution $ {\bf x}^{a}_+$ } If $a$ is a complex number with $\Re a>-1$, the function
\[{\bf x}^{a}_+= x^a, \;\;{\rm if}\;\; {\bf x}^{a}_+= 0 , \;\;{\rm if}\;\;  x\leq 0\]
is locally integrable so it defines a distribution. We give here only a hint and refer to \cite{hormander}, \cite{shilov}, \cite{kashiwara} for an extensive study from different sides. This distribution has two properties
\begin{equation}\label{equation1}
x{\bf x}^{a}_+= {\bf x}^{a+1}_+ \;\;{\rm if}\;\; \Re a>-1
\end{equation}
and
\begin{equation}\label{equation2}
\frac{d}{dx}{\bf x}^{a}_+= a{\bf x}^{a-1}_+  \;\;{\rm if}\;\; \Re a>0
\end{equation}
If $\Re a>0,\;{ x}_+^a$ is actually a continuous function, so it defines an hyperfunction (hence a distribution) and can be written as boundary values of holomorphic functions \cite{kashiwara} (p.83):
\begin{align*}
{\bf x}^{a}_+&= \frac{1}{e^{-i\pi \lambda}-e^{i\pi \lambda}}\left\{e^{-i\pi \lambda}(x+i0)^{\lambda}- e^{i\pi \lambda} (x-i0)^{\lambda}    \right\}\\
&=  \frac{1}{-2i\sin \pi \lambda}\left\{e^{-i\pi \lambda}(x+i0)^{\lambda}- e^{i\pi \lambda} (x-i0)^{\lambda}    \right\}.
\end{align*}
Before giving an outline of the analytical properties of the distribution we note that the $\Gamma$ function is a significant example of what we have in sight because
\[\Gamma(a)= \int_0^{\infty}e^{-x}x^{a-1}\,dx= \langle  {\bf x}^{a-1}_{+}, e^{-t} \rangle\]
and $a\mapsto \Gamma(a)  $ extends to a meromorphic function in $\BC$ with simple poles at $0,-1,-2,\cdots  $ with residues equal to $\displaystyle \frac{(-1)^n}{n!}  $. The methods leading to these results for this specific example work for a general study. To give an idea consider a function $\displaystyle \phi \in \mathcal{C}_0^{\infty}(\BR)  $, i.e. an  infinitely differentiable function with compact support, then for  any integer $N\geq 0$ and the Taylor coefficients $\displaystyle a_n= \frac{\phi^{(n)}(0)}{n!}$
\begin{align*}
\langle {\bf x}^{a-1}_{+}, \phi(t)\rangle&= \int_0^{\infty}\phi(x)x^{a}\,dt= \int_0^1 \phi(x)x^{a}\,dt+\int_1^{\infty}\phi(x)x^{a}\,dx\\
&=\int_0^{1}\left(\phi(x)-\sum_{n=0}^{N-1}a_nx^n\right)x^{a}\,dx+\sum_{n=0}^{N-1} \frac{a_n}{n+a}+\int_1^{\infty}\phi(x)x^{a}\,dx.
\end{align*}
Now from $\displaystyle  \phi(x)-\sum_{n=0}^{N-1}a_nx^n= {\rm O}(x^{N}) $, we deduce that the first integral define an analytic function in the right half-plane $\{\Re a>-N-1\}$,
the finite sum represents a meromorphic function with simple pole at  each $-n, 0\leq n\leq N-1$ and of residue $\displaystyle a_n$ and finally the last integral is an entire function. This argument shows that  for $\displaystyle  \phi \in \mathcal{C}_0^{\infty}(\BR)$, $\langle {\bf x}^{a-1}_{+}, \phi(t)\rangle $ is meromorphic in $\BC$ with simple pole at $-n$ for every non-negative integer $n$ and the corresponding residue is $\displaystyle a_n= \frac{\phi^{(n)}(0)}{n!}$. This argument is also valid for every
infinitely differentiable function $\phi$, of rapid decay at $+\infty$, as in the case of the $\Gamma$-function.

\subsection{The John transform of  ${\bf x}^{a}_{+}$}
Let $f\in {\mathcal S}(\BR^3)$ be a given function in the Schwartz space. As in \cite{gelfand2000} we define the John transform of $f$ over the line in $\BR^3$ of equation
\[x= \alpha t+\beta, \alpha \in \BR^3\setminus\{(0,0,0)\},\; \beta \in \BR^3\]
by
\[\phi(\alpha,\beta)= \int_{-\infty}^{+\infty}f(t\alpha+\beta)\;dt.\]
The definition forces $\phi$ to verify the following homogeneity and symmetry properties  \cite[p.~45]{gelfand2000}:
\[\sum_{i=1}^3 \alpha_i \frac{\partial \phi}{\partial \beta_i}=0,\quad \sum_{i=1}^3 \alpha_i \frac{\partial \phi}{\partial \alpha_i}=-\phi,\quad \phi(-\alpha,\beta)= \phi(\alpha,\beta).\]
If we represent lines in $\BR^3$ as common points of two affine hyperplanes
\begin{equation}\label{affine}
x_1= x_3\alpha_1+ \beta_1,\quad x_2= x_3\alpha_2+ \beta_2
\end{equation}
then the John transform gets the following form
\[\psi( \alpha_1,\alpha_2,\beta_1, \beta_2)= \int_{-\infty}^{+\infty}f(x_3\alpha_1+ \beta_1, x_3\alpha_2+ \beta_2 )\;dx_3.\]

To study the analytic continuation of the John transform of
 \[\displaystyle f_a(x_1,x_2,x_3)=  { x_1}^{a_1-1}_{+} { x_2}^{a_2-1}_{+} { x_3}^{a_3-1}_{+}, \quad a=(a_1,a_2,a_3)\in \BR^3\] 
 given by
 \[\psi_a( \alpha_1,\alpha_2,\beta_1, \beta_2)= \int_{-\infty}^{+\infty}(x_3\alpha_1+ \beta_1)^{a_1-1}_+( x_3\alpha_2+ \beta_2 )^{a_2-1}_+{x_3}_+^{a_3-1}\;dx_3.\]
 and by \eqref{hypergeometric2}:
  \begin{equation}\label{hypergeometric4}
 \psi_a( \alpha_1,\alpha_2,\beta_1, \beta_2)=\dfrac{\Gamma(a_2) \Gamma(a_3) }{\Gamma(a_2+a_3)}\beta_1^{a_1-1}\beta_2^{a_2+a_3-1}\vert \alpha_2\vert^{-a_3}
 {}_2 F_1\left(-a_1+1,a_3, a_2+a_3;\frac{\alpha_1\beta_2}{\alpha_2\beta_1}\right).
 \end{equation}
  We will also need a description of the modular $\lambda$ function in terms of even   zero-value (Nullwerte) Jacobi theta functions.  We first fix $x\in {\mathbb P}^1\setminus \{0,1,\infty\}$ and consider  $e_1,e_2,e_3\in \BC$ with \[e_1-e_3=1,\quad e_2-e_3=x,\quad e_1-e_2=1-x\] so that $e_1+e_2+e_3=0$ and this determine a  Weierstrass function $\displaystyle \wp(u)$, a meromorphic doubly periodic of periods $\omega_1$, $\omega_3$. We always assume  that $\tau=\dfrac{\omega_3}{\omega_1}$ has strictly positive imaginary part and $x$ appears as a $\Gamma(2)$-modular function of 
  $ \tau$, defined on the upper half-plane $\mathcal H$  with values in $\mathbb{P}^1\setminus\{0,1,\infty\}$. We set $v=\dfrac{u}{2\omega_1}$ and introduce the classical theta series: 
 \begin{equation}
 \label{6000}
 \begin{split}
 \displaystyle\vartheta_1(v,\tau)&=2q^{1/4}\sum_0^\infty(-1)^nq^{n(n+1)}\sin((2n+1)\pi v)\\
 \vartheta_2(v,\tau)&=2q^{1/4}\sum_0^\infty q^{n(n+1)}\cos((2n+1)\pi v)\\
 \vartheta_3(v,\tau)&=1+2\sum_0^\infty q^{n^2}\cos(2n\pi v)\\
\vartheta_4(v,\tau)&=1+2\sum_0^\infty(-1)^nq^{n^2}\cos(2n\pi v).
 \end{split}
 \end{equation}
 The Weierstrass function  $\wp$ is given in terms of Jacobi theta functions by:
  \begin{equation}
 \label{5498}
 \begin{split}
 \displaystyle\sqrt{\wp(u)-e_1}&= \sqrt[4]{1-x}\; \dfrac{\vartheta_2(v,\tau)}{\vartheta_1(\chi,\tau)}\\
 \sqrt{\wp(u)-e_2}&= \sqrt[4]{x(1-x)}\; \dfrac{\vartheta_3(v,\tau)}{\vartheta_1(v,\tau)}\\
 \sqrt{\wp(u)-e_3}&= \sqrt[4]{x}\;\dfrac{\vartheta_4(v,\tau)}{\vartheta_1(v,\tau)}.
 \end{split}
 \end{equation}
 If we set $u=0$ one obtains \cite[p. 96-97]{dolgachev}:
 \[x= \dfrac{e_2-e_3}{e_1-e_3}= \dfrac{\vartheta_2(0,\tau)^4}{\vartheta_3(0,\tau)^4}= \lambda(\tau), \quad 1-x=  \dfrac{\vartheta_1(0,\tau)^4}{\vartheta_3(0,\tau)^4}.\]
  We have from \eqref{hypergeometric3} and  for $0<\Re b<\Re c$ the Wirtinger representation \cite{kampe}:
 \begin{equation}
 \label{5499}
\displaystyle \dfrac{1}{2}\Gamma(b)\Gamma(c-b)F(a,b,c;x(\tau))=\pi^{2b}\Gamma\left(c\right)\vartheta_3(0,\tau)^{4b}\int_0^{\frac{1}{2}}\Phi(v,\tau)d\nu,
\end{equation}
with
$$ \Phi(v,\tau)=\left(\dfrac{\vartheta_1(v,\tau)}{\vartheta_1(0,\tau)}\right)^{2b-1}\left(\dfrac{\vartheta_2(v,\tau)}{\vartheta_2(0,\tau)}\right)^{2(c-b)-1}\left(\dfrac{\vartheta_3(v,\tau)}{\vartheta_3(0,\tau)}\right)^{1-2a}\left(\dfrac{\vartheta_4(v,\tau)}{\vartheta_4(0,\tau)}\right)^{1-2(c-a)}.$$
By using \eqref{hypergeo}, \eqref{hypergeometric4} and \eqref{5499} we obtain one of our main results: The John transform \eqref{hypergeometric4}  is expressible using theta functions, with $\displaystyle x(\tau)=   \dfrac{\alpha_1\beta_2}{\alpha_2\beta_1} . $
\begin{remark}
A natural question concerns the meaning of the factor $\displaystyle \dfrac{\alpha_1\beta_2}{\alpha_2\beta_1}   $ in \eqref{hypergeometric4}. If we eliminate $x_3$ in the equations of the line $\mathcal L$: $$x_1= x_3\alpha_1+ \beta_1,\quad x_2= x_3\alpha_2+ \beta_2$$ we obtain the line $\mathcal L'$:
\[\frac{x_1}{\beta_1}-1= \dfrac{\alpha_1\beta_2}{\alpha_2\beta_1} \left(\frac{x_2}{\beta_2}-1\right)\]
which is the projection of the line  $\mathcal L$ in the $(x_1,x_3)$-plane. The affine map
\[(x_1,x_2)\longrightarrow \left( \frac{x_1}{\beta_1}-1,   \frac{x_2}{\beta_2}-1\right)\]
transforms $\mathcal L'$ into a line of slope $\displaystyle \dfrac{\alpha_1\beta_2}{\alpha_2\beta_1}    $.
\end{remark}
 \subsection{Eigenvalue problem}
We investigate in this section the eigenvalue problem for the differential operators associated to the Frobenius determinant of a finite abelian group. We recall that this differential operator can be factorized into
$$\displaystyle\Theta(G)\left((\partial_{g})_{g\in G}\right)=\prod_{g\in G}\partial_{g},\quad \partial_g=\dfrac{\partial}{\partial x_g}.$$
We are looking for holomorphic solutions $U(x_g,g\in G)$ to
\begin{equation}
\label{f5400}
\displaystyle\Theta(G)\left((\partial_{g})_{g\in G}\right)U=\prod_{g\in G}\partial_{g}U=\lambda U,\quad \lambda\in \C^\times.
\end{equation}
We remark that by an elementary change of variables, the equation \eqref{f5400} is equivalent to the equation
\begin{equation}
\label{f5401}
\displaystyle\Theta(G)\left((\partial_{g})_{g\in G}\right)U=\prod_{g\in G}\partial_{g}U= U.
\end{equation}
Let $n=\sharp{G}$. For ease of notation we identify $(x_g)_{g\in G}$ with $(x_i)_{1\leqslant i\leqslant n}$. Let us introduce the following series which generalize the Bessel functions
\begin{equation}
\label{f5402}
\displaystyle E_k(z)=\sum_{m\geqslant 0}\dfrac{z^m}{(m!)^k},\quad k\geqslant2.
\end{equation}
\begin{theorem}[\cite{chaundy}]
The eigenvectors $U$ of the operator $\dfrac{\partial^n}{\partial_{x_1}\ldots\partial_{x_n}}$ with eigenvalue $1$ can be expressed in the form
\begin{equation}
\label{f5403}
\begin{split}
\displaystyle &U(x_1,\ldots,x_n)=\\
&\sum_{i=1}^n\int_0^{x_1}\int_0^{x_2}\ldots\hat{\int_0^{x_i}}\ldots \int_0^{x_n}\phi_i(t_1,\ldots,t_{n-1})E_{n}(x_i(x_1-t_1)\ldots(x_{i-1}-t_{i-1})\\
&\hat{(x_i-t_i)}(x_{i+1}-t_{i+1})\ldots(x_n-t_n))dt_1dt_2\ldots dt_{n-1}+\sum_{1\leqslant i< j\leqslant n}^n\int_0^{x_1}\int_0^{x_2}\ldots\hat{\int_0^{x_i}}\ldots \hat{\int_0^{x_j}}\ldots \int_0^{x_n}\\
&\varphi_{i,j}(v_1,\ldots,v_{n-2})E_n(x_ix_j(x_1-v_1)(x_2-v_2)\ldots\hat{(x_i-v_i)}\ldots\hat{(x_j-v_j)}\\
&\ldots(x_n-v_n))dv_1dv_2\ldots dv_{n-2}+\ldots+CE_n(x_1x_2\ldots x_n),
\end{split}
\end{equation}
where $C\in \C$ and $\hat{}$ of a symbol means that it is suppressed from the expression.
\end{theorem}
A large part of the theory of Bessel functions can be built on an operator formalism. In fact the two fundamental relations:
\begin{equation}
\label{5500}
\displaystyle \dfrac{2\nu}{z}J_{\nu}(z)=J_{\nu-1}(z)+J_{\nu+1}(z)
\end{equation}
\begin{equation}
\label{5501}
\displaystyle 2\dfrac{d}{dz}J_{\nu}(z)=J_{\nu-1}(z)-J_{\nu+1}(z)
\end{equation}
not only give the second order differential equation for the Bessel function
$$J^{\prime\prime}_\nu(z)+\dfrac{1}{z}J_\nu^\prime(z)+(1-\dfrac{\nu^2}{z^2})J_\nu(z)=0$$
but give also a representation of the operator $M$ of multiplication by $\dfrac{2}{z}$ in terms of the shift operators $\tau_{-1}$ and $\tau_{1}$, on the right and on the left of the indices. The second relation gives a representation of the operator $2\dfrac{d}{dz}$ by the same shift operators. Hence for instance
$$\displaystyle \left(\dfrac{1}{z}\dfrac{d}{dz}\right)^m\left[z^\nu J_\nu(z)\right]=z^{\nu-m}J_{\nu-m}(z)$$
$$\left(\dfrac{1}{z}\dfrac{d}{dz}\right)^m\left[z^{-\nu}J_{\nu}(z)\right]=(-1)^mz^{-\nu-m}J_{\nu+m},\quad m=1,2,\ldots.$$
Moreover $y=z^{-n}J_n(z)$ satisfies:
$$\displaystyle z\dfrac{d^2y}{dz^2}+(2n+1)\dfrac{dy}{dz}+zy=0.$$
We define a formal operator $D^{-1}$, acting on continuous functions by 
$$\displaystyle D^{-1}f(z)=\int_0^zf(t)dt$$
which is an inverse of the operator $\displaystyle D= \dfrac{d}{dz}$. 
For $\nu=m$ an integer we have 
$$J_m(z)=z^{-m}\left(\dfrac{1}{z}\dfrac{d}{dz}\right)^{-m}J_0(z)$$
so that for $|z|<1$ and interpreting $\left(\dfrac{1}{z}\dfrac{d}{dz}\right)^{-1}$ as the product of the two inverse operators $ D^{-1}\circ z$ we get for example 
$$\displaystyle J_1(z)=\dfrac{1}{z}\left(\dfrac{1}{z}\dfrac{d}{dz}\right)^{-1}zJ_0(z)=\dfrac{1}{z}\int_0^z \zeta J_0(\zeta)d\zeta.$$
For $m=2$ we obtain:
\begin{equation}
\label{5502}
\displaystyle
\begin{split}
J_2(z)&=\dfrac{1}{z^2}\left(\dfrac{1}{z}\dfrac{d}{dz}\right)^{-1}\left(\dfrac{1}{z}\dfrac{d}{dz}\right)^{-1}J_0(z)\\
&=\dfrac{1}{z^2}\left(\dfrac{1}{z}\dfrac{d}{dz}\right)^{-1}zJ_1(z)=\dfrac{1}{z^2}\int_0^z\zeta^2 J_1(\zeta)d\zeta\\
&=\dfrac{1}{2z^2}\int_0^z(z^2-\zeta^2)J_0(\zeta)d\zeta.
\end{split}
\end{equation}
The last equality is obtained by integration by parts. By induction we have a representation of all $J_m$ in terms of the single $J_0$:
$$\displaystyle J_m(z)=\dfrac{1}{2^{m-1}\Gamma(m)z^m}\int_0^z(z^2-\zeta^2)^{m-1}J_0(\zeta)d\zeta.$$
The change of variables $\zeta=z\sin(\theta)$ gives
$$\displaystyle J_m(z)=\dfrac{z^m}{2^{m-1}\Gamma(m)}\int_0^{\pi/2}\sin\theta\cos\theta^{2m-1}J_0(z\sin\theta)d\theta.$$
By a generalization to arbitrary indices we obtain the well-known Sonine's integral
$$\displaystyle J_m(z)=\dfrac{z^\nu}{2^{\nu-1}\Gamma(\nu)}\int_0^{\pi/2}\sin\theta\cos\theta^{2\nu-1}J_0(z\sin\theta)d\theta.$$
Another aspect of operator formalism is related to the symmetric products or tensor products of differential operators:
$J_\nu(z)$ satisfies the differential equation:
$$L_\nu(J_\nu)=J_\nu^{\prime\prime}(z)+\dfrac{1}{z}J_\nu^\prime(z)+(1-\dfrac{\nu}{z^2})J_\nu(z)=0$$
and likewise $J_\mu$ satisfies
$$L_\mu(J_\mu)=J_\mu^{\prime\prime}(z)+\dfrac{1}{z}J_\mu^\prime(z)+(1-\dfrac{\mu}{z^2})J_\mu(z)=0$$
and then $J_\mu(z)J_\nu(z)$ satisfies a fourth order differential equation. More precisely a generalization of the Gauss hypergeometric function \eqref{hypergeometric1} 
leads to
$$\displaystyle\Hypergeometric{p}{q}{a}{b}{z}=\sum_0^\infty\dfrac{(a_1)_n(a_2)_n\ldots(a_p)_n}{(b_1)_n\ldots(b_q)_n}\dfrac{z^n}{n!}$$
where $(a)_0=1$, $(a)_n=a(a+1)\ldots(a+n-1)=\dfrac{\Gamma(a+n)}{\Gamma(a)}$. Let $\delta=z\dfrac{d}{dz}$, then $\Hypergeometric{p}{q}{a}{b}{z}$ satisfies
$$\{\delta(\delta+b_1-1)\ldots(\delta+b_q-1)-z(\delta+a_1)\ldots (\delta+a_p)\}u=0$$
which can be written as:
$$\displaystyle \sum_1^q z^{n-1}(a_nz-b_n)D^nv+a_0v+z^qD^{q+1}v=0,\quad D=\dfrac{d}{dz}.$$
For $q=3$ we obtain a fourth order differential operator satisfied by the generalized hypergeometric function $_2F_3$. This differential equation is precisely the tensor product of the two second order differential operators $L_\nu$ and $L_\mu$. Hence the following formulas
\begin{equation}
\label{5503}
\displaystyle
J_\nu(x)J_\mu(x)=(\dfrac{1}{2}z)^{\nu+\mu}\sum_0^\infty\dfrac{(-1)^m\Gamma(\mu+\nu+2m+1)}{\Gamma(\mu+m+1)+\Gamma(\nu+m+1)\Gamma(\mu+\nu+m+1)}(\frac{1}{2}z)^{2m}
\end{equation}
and
\begin{equation}
\label{5504}
\displaystyle
\begin{split}
&\Gamma(\mu+1)\Gamma(\nu+1)J_\nu(x)J_\mu(x)\\
&=(\dfrac{1}{2}z)^{\nu+\mu}\,_2F_3\left(\frac{1}{2}+\frac{1}{2}\mu+\frac{1}{2}\nu,\frac{1}{2}+\frac{1}{2}\mu+\frac{1}{2}\nu;1+\nu,1+\mu,1+\nu+\mu;-z^2\right).
\end{split}
\end{equation}
Note that $J_\nu(z)=(\dfrac{1}{2}z)^\nu\dfrac{e^{-iz}}{\Gamma(\nu+1)}\,_1F_1(\nu+\frac{1}{2};2\nu+1;2iz)$.
It is possible to expand an analytic function $f$ in a domain containing the origin in a series of Bessel functions. This is Neumann's series and has the form
$$\displaystyle f(z)=\sum_0^\infty a_nJ_n(z),\quad a_n\in\C.$$
Using Cauchy's formula, we are reduced to consider an expansion of the form
\begin{equation}
\label{5505}
\dfrac{1}{t-z}=\Theta_0(t)J_0(z)+2\Theta_1(t)J_1(z)+2\Theta_2(t)J_2(z)+\ldots
\end{equation}
where $\Theta_n(t)$ are functions of $t$ alone. The main ingredient is the following partial differential equation
$$(\dfrac{\partial}{\partial t}+\dfrac{\partial}{\partial z})\dfrac{1}{t-z}=0$$
and from computations and by using the relation
$$2J_n^\prime(z)=J_{n-1}(z)-J_{n+1}(z)$$
we obtain
$$\Theta_0(t)=\dfrac{1}{t},\quad\Theta_1(t)=-\Theta_0^\prime(t),\quad\Theta_{n+1}(t)=\Theta_{n-1}(t)-2\Theta_n^\prime(t)$$
is a remarkable that the equations satisfied par $J_n$ and $\Theta_n$ are adjoint. Actually $\Theta_n(t)$ is a polynomial in $1/t$ given by:
$$\displaystyle\Theta_{2n}(t)=\frac{1}{2}n\sum_0^n\dfrac{(n+m-1)!}{(n-m)!}(1/2t)^{-2m-1}$$
$$\displaystyle\Theta_{2n}(t)=\frac{1}{2}(n+\frac{1}{2})\sum_0^n\dfrac{(n+m)!}{(n-m)!}(1/2t)^{-2m-2}.$$
They are called Neumann's polynomials. Writing both cases as
$$\displaystyle\Theta_n(t)=1/4\sum_0^{n/2}\dfrac{n(n-m+1)!}{m!}(1/2t)^{2m-n-1}$$
we see that
$$|\Theta_n(t)|\leqslant 2^{n-1}n!|t|^{-n-1}e^{\frac{1}{4}|t|^2}.$$
If $\mathscr{C}$ is a simple closed contour around the origin, then
$$\displaystyle \int_\mathscr{C}\Theta_m(z)\Theta_n(z)dz=0\quad\quad\quad m=n;\,m≠n.$$
$$\displaystyle \int_\mathscr{C}J_m(z)\Theta_n(z)dz=0\quad\quad\quad m≠n.$$
$$\displaystyle \int_\mathscr{C}J_m(z)\Theta_m(z)dz=i\pi.$$
\begin{remark}
If we consider the $\C$-vector space generated by all $J_n,\,n\in\Z$, with 
$$J_{-n}(z)=(-1)^nJ_n(z),\quad n\in\Z$$
and if $l$ (resp. $r$) is the shift to the left (resp. right) acting on the indices, the relation \eqref{5501} reads
$$(2\dfrac{d}{dz})J_\nu=(l-r)J_\nu$$
with $lr=rl=\Id$. Then
$$\displaystyle(2\dfrac{d}{dz})^m=(l-r)^m=\sum_0^m\binom{m}{k}(-1)^{m-k}l^kr^{m-k}$$
or
$$\displaystyle (2\dfrac{d}{dz})^mJ_\nu=\sum_0^m\binom{m}{k}(-1)^{m-k}J_{\nu-2k+m}.$$
\end{remark}\section{Factorization, differential Galois theory and Bessel functions}
We give only a short overview of differential Galois theory. We refer to \cite{sing2003} and \cite[chap.~8]{beukers2001} for more details. Let $(k,\partial)$ be a differential field with algebraically closed field of constants $C_k$ and consider the linear differential equation
\begin{equation}
\label{5700}
\mathscr{L}y=0,\quad\mathscr{L}=\partial^n+f_1\partial^{n-1}+\ldots+f_n,\quad f_i\in k,\quad (i=1,\ldots,n).
\end{equation}
Usually, the solutions of \eqref{5700} do not lie in $k$. So we look for differential extensions of $(k,\partial)$ containing the solutions.
\begin{definition}
A differential extension of $(\mathscr{K},\partial)$ of $(k,\partial)$ such that
\begin{enumerate}  
\item $C_{\mathscr{K}}=C_k$,
\item $\mathscr{K}$ contains $n$ $C_{k}$ linear independent solutions $y_1,\ldots,y_n$ of \eqref{5700} and 
$$\mathscr{K}=k<y_1,\ldots,y_n>,$$
where $k<y_1,\ldots,y_n>$ is the differential extension generated by $(y_1,\ldots,y_n)$, is called a Picard-Vessiot extension or shortly PVE.
\end{enumerate}
\end{definition}
\begin{theorem}[Kolchin]
Let $(k,\partial)$ be a differential field. Assume the characteristic of $k$ is zero and that $C_k$ is algebraically closed. Then to any linear differential there exists an associated PVE. Moreover, this extension is unique up to differential isomorphism.   
\end{theorem}
\begin{definition}
Let $(\mathscr{K},\partial)$ be a PVE of equation \eqref{5700}. The differential Galois group of equation \eqref{5700} (or of $\mathscr{K}/k)$ is defined as the group of differential automorphisms $\phi:\mathscr{K}\to\mathscr{K}$ such that $\phi f=f$ for all $f\in k$; notation: $Gal_\partial(\mathscr{K}/k)$.
\end{definition}
\begin{lemma}
Let notations be as above, and let $\mathscr{J}$ be the ideal of polynomials $$Q\in k[X_1^{(0)},\ldots,X_n^{(0)},X_1^{(1)},\ldots,X_n^{(1)},\ldots,X_1^{(n-1)},\ldots,X_n^{(n-1)}]$$ such that $Q(\ldots,\partial^jy_i,\ldots)=0$ (we have substituted $\partial^jy_i$ for $X_i^{(j)}$). Then 
$$\displaystyle Gal_\partial (\mathscr{K}/k)=\big\{g_{ik}\in GL(n,C_k)|Q(\ldots,\partial^j(\sum_{k=1}^ng_{ik}y_k),\ldots)=0,\forall Q\in \mathscr{J}\big\}.$$
\end{lemma}
\begin{remark}
The determinant of $\phi\ Gal_\partial(\mathscr{K}/k)$ can be read off from its action on the wronskian $W(y_1,\ldots,y_n)$, since $W(\phi y_1,\ldots,\phi y_n)=\det(\phi)W(y_1,\ldots,y_n)$. Moreover notice that $W(vy_1,\ldots,vy_n)=v^rW(y_1,\ldots,y_n),v\in\mathscr{K}$ and $\partial W(y_1,\ldots,y_n)=-f_1W(y_1,\ldots,y_n)$. 
\end{remark}
\begin{theorem}[Kolchin]
Let $(\mathscr{K},\partial)$ be a PVE extension of $(k,\partial)$ with differential Galois group $G=Gal_\partial (\mathscr{K}/k)$. Then 
\begin{enumerate}
\item If $f\in\mathscr{K}$ is such that $\phi f=f,\forall\in G$, then $f\in\mathscr{K}$.
\item Let $H$ be an algebraic subgroup of $G$ such that $k=\big\{f\in\mathscr{K}|\phi f=f,\phi\in H\big\}$. Then $H=G$.
\item There is a one-to-one correspondence  between algebraic subgroups $H\subset G$ and intermediate differential extensions $M$ of $k$ (i.e $k\subset M\subset \mathscr{K})$ 
$$H= Gal_\partial(\mathscr{K}/M),M=\{f\in\mathscr{K}|\phi f=f,\forall \phi\in H\}.$$
\item Under the correspondence given by $(3)$ a normal algebraic subgroup $H$ of $G$ corresponds to a PVE and conversely. In such a situation $Gal_\partial(M/k)=G/H$.
\item The differential Galois group of a PVE is a linear algebraic group over $C_k$ with dimension as an algebraic group equal to the transcendence degree of $\mathscr{K}$ over $k$. 
\end{enumerate}
\end{theorem}
\begin{theorem}
Let $(\mathscr{K}/k)$ be a PVE corresponding to \eqref{5700}. Let $V$ be is $C_k$ vector space of solutions and $G=Gal_\partial(\mathscr{K}/k)$ the differential Galois group. Then the following statements are equivalent:
\begin{enumerate}
\item There is a non-trivial linear space $W\subset V$ which is stable under $G$.
\item The operator $\mathscr{L}$ factors as $\mathscr{L}_1\mathscr{L}_2$, where $\mathscr{L}_1$ and $\mathscr{L}_2$ are linear differential operators with coefficients in $F$ and of order strictly less than $n$. Moreover $y\in W\iff\mathscr{L}_2y=0.$ 
\end{enumerate}
\end{theorem}
\begin{definition}
Equation \eqref{5700} is called irreducible over $k$ if $\mathscr{L}$ does not factor over $F$. So the previous theorem implies that equation \eqref{5700} is irreducible $\iff$ its differential Galois group acts irreducibly on the space of solutions.
\end{definition}
\begin{corollary}
Let notations be as in the previous theorem. Then $G^\circ$, the connected component of the identity in $G$ acts irreducibly if and only if $\mathscr{L}$ does not factor over any finite extension of $k$. 
\end{corollary}
 \begin{definition}
 Let $(k,\partial)$ be a differential field having an algebraically closed field of constants $C_k$. An extension $K$ of $k$ is a Liouvillian extension of $k$ if the field of constants of $K$ is $C$ and if there exists a tower of fields 
 $$k=K_0\subset K_1\subset\ldots\subset K_n$$
 such that $K_i=K_{i-1}(t_i)$ for $i=1,\ldots,n$ where either 
 \begin{enumerate}
 \item $t_i^\prime\in K_{i-1}$.
 \item $t_i≠0$ and $t_i^\prime/t_i\in K_{i-1}$.
 \item $t_i$ is algebraic over $K_{i-1}$.
 \end{enumerate}
 \end{definition}
 The main result in this topic (concerning Liouvillian extensions) is that if $\mathscr{K}$ is a Picard-Vessiot extension of $k$ with a differential Galois group $G$ then the $3$ following statements are equivalent (see \cite[p.~34]{sing2003})
 \begin{theorem}
 \label{liou}
 Let $K$ be Picard-Vessiot extension of $k$ with differential Galois group $G$. The following are equivalent
 \begin{enumerate}
 \item $G$ or equivalently $G^\circ$, the connected component of the identity, is a solvable group.
 \item $\mathscr{K}$ is a Liouvillian extension of $k$.
 \item $\mathscr{K}$ is contained in a Liouvillian extension of $k$.
 \end{enumerate}
 \end{theorem}
 Let us now apply these concepts to the case of the second order linear homogeneous differential equation over $(\C(z),\dfrac{d}{dz}:=^\prime)$
 \begin{theorem}
 Consider 
 \begin{equation}
 \label{5701}
 y^{\prime\prime}+P(z)y^\prime+Q(z)y=R(z), 
 \end{equation}
 where $P(z),Q(z),R(z)$ are elementary functions, i.e, lie in a Liouvillian extension of $\C(z)$. If there is a non-zero elementary function $y_1(z)$, which is a solution to $y^{\prime\prime}+P(z)y^\prime+Q(z)y=0$, then every solution of equation \eqref{5701} is an elementary function.
 \end{theorem}
 \begin{proof}
 We consider \eqref{5701} in a region where $y_1(z)$, $P(z)$, $Q(z)$ and $R(z)$ are holomorphic. An independent solution of the homogeneous equation associated to \eqref{5701} is 
 $$\displaystyle y_2(z)=y_1(z)\int\exp[\int-\frac{2y^\prime_1+Py_1}{y_1}dz]dz.$$
 Thus $y_2$ is elementary. The formula of variation of parameters yields then an elementary solution $y(z)$ of \eqref{5701} and thus every solution of \eqref{5701} is elementary.
 \end{proof}
 \begin{corollary}
 If $u^\prime+u^2+P(z)u+Q(z)=0$ has one elementary solution, then every solution of it is elementary.
\end{corollary}
\begin{proof}
Let $u_1(z)$ be an elementary solution of the Riccati equation and let $y_1(z)=e^{\int u_1dz}$. Then $y_1$ is a non-zero solution of the linear differential equation 
$$y^{\prime\prime}+P(z)y^\prime+Q(z)y=0$$
and each solution of the linear equation is elementary. Let $u_2(z)$ be a solution a solution of the Riccati equation. Then $u_2(z)=\dfrac{y_2^\prime}{y_2}$, where $y_2$ is an elementary solution of the linear equation. Hence $u_2(z)$ is elementary.
\end{proof}
\begin{theorem}[Bernoulli]
\label{ber1}
The Bessel equation $z^2y^{\prime\prime}+zy^\prime+(z^2-\nu^2)y=0$ with $(2\nu=\pm1,\pm3,\pm5,\ldots)$ an odd integer has only elementary solutions.
\end{theorem}
\begin{proof}
Consider the Bessel equation with parameters $2\nu=1,3,5,\ldots$ (the statement is symmetric or even in $\nu$). Put $u(z)=z^{1/2}y(z)$ and write also $z=ix$ to get
$$\dfrac{d^2u}{dx^2}=[1+\dfrac{p(p+1)}{x^2}]u$$
where $p=\nu-1/2\geq0$ is an integer. The case $p=0$ is known. In this case a basis of solutions is given in an elementary way with the trigonometric functions. So we assume that $p\geq1$. Now let $u(x)=e^xx^{-p}\varphi(x)$ to get
$$\dfrac{d^2\varphi}{dx^2}+2(1-\dfrac{p}{x})\dfrac{d\varphi}{dx}-2p\dfrac{\varphi }{x}=0.$$
We show that there exists a polynomial solution $\varphi(x)$ to that equation and hence $y(z)=z^{-1/2}e^{-iz}(-iz)^{-p}\varphi(-iz)$ is an elementary solution of Bessel's equation. We find a power series solution $\varphi(x)=a_0+a_1x+\ldots+a_mx^m+\ldots$ and choose $a_0=1$ and $a_m=0,\,m\geq p+1$.  
\end{proof}
The proofs of the following three statements are omitted because they are rather technical and not very enlightening. But they are straightforward.
\begin{theorem}
\label{ber4}
Let $P(z)≠0$ be algebraic over $\C(z)$. If the Riccati equation $y^\prime+y^2=P(x)$ has an elementary solution, then it also has an algebraic solution.
\end{theorem}
\begin{lemma}
\label{ber2}
Each algebraic solution of $\dfrac{dv}{dz}+v^2=1+\dfrac{p(p+1)}{2}$, for a complex constant $p$ is rational.
\end{lemma}
\begin{lemma}
\label{ber3}
If $\dfrac{dv}{dz}+v^2=1+\dfrac{p(p+1)}{2}$ has a rational function solution, then $p$ is an integer.
\end{lemma}
\begin{theorem}
The Bessel equation $z^2y^{\prime\prime}+zy^\prime+(z^2-\nu^2)y=0$ has no (non-zero) elementary function solution if $2\nu$ is not an odd integer $(2\nu≠\pm1,≠\pm3,≠\pm5,\ldots)$. Hence its differential Galois group is a non-solvable algebraic subgroup of $SL(2,\C)$.
\end{theorem}
\begin{proof}
If $y(z)≠0$ is an elementary solution of the Bessel equation, then set $u(z)=z^{1/2}y(z),\,z=ix$ to define $u(x)≠0$, an elementary solution of $\dfrac{d^2u}{dx^2}=[1+\dfrac{p(p+1)}{x^2}]u$, where $p=\nu-\frac{1}{2}$. Define $v(x)=\dfrac{u^\prime(x)}{u(x)}$ an elementary solution of $\dfrac{dv}{dx}+v^2=1+\dfrac{p(p+1)}{x^2}$. This impossible since from \theoref{ber4} this last Riccati equation must admit an algebraic solution over $\C(x)$. This algebraic solution is in fact a rational solution by \lemref{ber2} and this requires that $p$ be an integer by \lemref{ber3}. The statement about the differential Galois group follows from the fact that the wronskian of any basis of solutions to a Bessel equation belongs to $\C(z)$ and from the given proof that the Bessel equation with parameter $\nu:(2\nu=≠\pm1,≠\pm3,≠\pm5,\ldots)$ does not have Liouvillian solutions and from \theoref{liou}.
\end{proof}
In the following we adapt an argument given in \cite{beukers2001} in the case $n=0$.
\begin{theorem}
\label{th2010}
Consider a Bessel differential equation with integer parameter $n\in\Z$: $z^2y^{\prime\prime}+zy^\prime+(z^2-n^2)y=0$. Then the associated differential operator is irreducible. This means that its differential Galois group acts irreducibly on its solution space. Moreover it is precisely $SL(2,\C)$.
\end{theorem}
\begin{proof}
This follows from the fact that a basis of solution around $0$ is given by $J_n$ and $J_n\log(z)+Y_n$ with $J_n$ the classical Bessel function and $Y_n$ some power series. Indeed if the differential Galois group $G$ acts reducibly, then there exists a solution $y$ such that $\phi:y\mapsto\lambda(\phi)y,\,\lambda(\phi)\in\C$ for any $\phi\in G$. This implies that $\phi^\prime/\phi\in\C(z)$. This certainly is not possible if $\phi$ contains $\log(z)$. Hence we can take $y=J_n(z)$ and $J_n^\prime/J_n\in\C(z)$. Again this not possible because $J_n(z)$ has infinitely many zeros. So $G$ acts irreducibly. Notice that $J_n$ is transcendental over $\C(z)$ and $J_n\log(z)+Y_n$ is transcendental over $\C(z,J_n(z))$ for the simple fact that $\log(z)$ is transcendental over the field of Laurent series in $z$. Hence the transcendence degree of the PVE $(\mathscr{K}/\C(z))$ is at least $2$. Besides we know from the fact the wronskian belongs to the ground field $\C(z)$ that the differential Galois group $G$ lies in $SL(2,\C)$; all these facts imply $G=SL(2,\C)$.
\end{proof}
 Consider 
 $$\left\{x^2\dfrac{d^2}{dx^2}+x\dfrac{d}{dx}+(x^2-n^2)\right\}y=0.$$
 We have seen in \theoref{th2010} that the associated differential operator is not factorable. But as in quantum mechanics (Harmonic oscillator) we can try to write it in the form
 $$\left\{\dfrac{d^2}{dx^2}+\dfrac{1}{x}\dfrac{d}{dx}+(1-n^2/x^2)\right\}=\left\{\dfrac{d}{dx}-\varphi_1(x)\right\}\left\{\dfrac{d}{dx}-\varphi_2(x)\right\}+\lambda(x)=0,\,\lambda\in \C(x).$$
 This gives
 \begin{align*} \varphi_1+\varphi_2&=-1/x\\
 \lambda+\varphi_1\varphi_2-\varphi_2^\prime&=1-n^2/x^2
 \end{align*}
or  
\[\varphi_2^\prime+\frac{1}{x}\varphi_2+\varphi_2^2-(\lambda-1)-n^2/x^2=0.\] 
If $\lambda=1$, then $\varphi_2=\pm n/x$ and 
\begin{align*}\varphi_2(x)&=\pm n/x\\\varphi_1(x)&=\dfrac{\mp n-1}{x}.\end{align*}.
 Choosing the sign $+$ for $\varphi_1(x)$ and therefore $-$ for $\varphi_2(x)$ yields
  $$\left\{(\dfrac{d}{dx}-\dfrac{n-1}{x})(\dfrac{d}{dx}+\dfrac{n}{x})+1\right\}J_n(x)=0$$
 while choosing the sign $-$ for $\varphi_1(x)$ and $+$ for $\varphi_2(x)$ yields another relation
 $$\left\{(\dfrac{d}{dx}+\dfrac{n+1}{x})(\dfrac{d}{dx}-\dfrac{n}{x})+1\right\}J_n(x)=0.$$
 This also means 
 \begin{equation}\label{5800}\begin{matrix}(\dfrac{d}{dx}-\dfrac{n-1}{x})(\dfrac{d}{dx}+\dfrac{n}{x})J_n(x)&=-J_n(x)\\\\(\dfrac{d}{dx}+\dfrac{n+1}{x})(\dfrac{d}{dx}-\dfrac{n}{x})J_n(x)&=-J_n(x). \end{matrix}.\end{equation}
 The relations \eqref{5800} hold for instance when$$(\dfrac{d}{dx}+\dfrac{n}{x})J_n(x)=J_{n-1}(x),\quad (\dfrac{d}{dx}-\dfrac{n}{x})J_n(x)=-J_{n+1}(x).$$
If we eliminate the derivative term between the latter equations we get $2nJ_n(x)-xJ_{n-1}-xJ_{n+1}=0.$

Recall that the Laplace transform of a suitable function $F$ is given by
$$\displaystyle\mathscr{L}(F)(s)=f(s)=\int_0^{+\infty}e^{-sx}F(x)dx.$$
We have 
$$\mathscr{L}(F^{(k)})(s)=s^kf(s),\quad \mathscr{L}(x^kF(x))=(-1)^kf^{(k)}(s)=(-1)^k \mathscr{L}(F)^{(k)}(s).$$
If $j_0$ is the transform of $J_0$, then $(s^2+1)j^\prime_0+sj_0(s)=0$, the solution of which is $$\displaystyle j_0(s)=\sqrt{s^2+1}=s^{-1}(1+s^2)^{-1/2}=\sum_0^\infty\binom{-1/2}{k^2}s^{-(2k+1)}.$$
Using $\mathscr{L}(x^k)=\dfrac{k!}{s^{k+1}}$, we obtain:
$$\displaystyle J_0(x)=\sum_0^\infty\binom{-1/2}{k^2}\dfrac{x^{2k}}{(2k)!}$$ 
and after simplification we get 
$$\displaystyle J_0(x)=\sum_0^{+\infty}\dfrac{(-1)^k}{(k!)^2}\left(\dfrac{x}{2}\right)^{2k}.$$
Now from $(\dfrac{d}{dx}-\dfrac{n}{x})J_n(x)=J_{n+1}(x)$ and introducing the integrating factor $\exp(-\int\frac{n}{x}dx)=x^{-n}$, we obtain
$$\dfrac{1}{x}\dfrac{d}{dx}\left(\dfrac{J_n(x)}{x^n}\right)=-\dfrac{J_{n+1}(x)}{x^{n+1}}.$$
Repeated application of $\dfrac{1}{x}\dfrac{d}{dx}$ leads to $(\dfrac{1}{x}\dfrac{d}{dx})^m(\dfrac{J_n(x)}{x^n})=(-1)^m(\dfrac{J_{n+m}(x)}{x^{n+m}})$
or Rayleigh's formula:
$$J_{n+m}(x)=(-1)^mx^{n+m}(\dfrac{1}{x}\dfrac{d}{dx})^m\dfrac{J_n(x)}{x^n}$$
and
$$J_{m}(x)=(-1)^mx^{m}(\dfrac{1}{x}\dfrac{d}{dx})^mJ_0(x).$$
\section{Borel transform}
A general question in analytic function theory is to find an explicit analytic continuation of a power series $f(z)=\displaystyle \sum_{n\geqslant0}a_nz^n$ of radius of convergence $R$, $0<R<+\infty$. The largest domain on which $f$ can be continued analytically is star-shaped with respect to the origin. This is the main point of the Borel transform. For a power series as above we define
$$\displaystyle\mathscr{B}(f)(z)=\sum_{n\geqslant 0}\dfrac{a_n}{n!}z^n.$$
Then $\mathscr{B}(f)$ is an entire function of exponential type $1/R$ and we have the integral representation on the disc $D(0,R)=\{z\in\C,|z|<R\}$
$$\displaystyle f(z)=\int_0^\infty e^{-t}\mathscr{B}(f)(tz)dt.$$
The idea is actually very simple for the assumption $\displaystyle\limsup_{n\to+\infty}\sqrt[n]{|a_n|}=1/R$ gives for any $\epsilon\bg 0$ and $n\geqslant n_\epsilon$:
$$\dfrac{|a_nz^n|}{n!}\leqslant(1/R+\epsilon)\dfrac{|z|^n}{n!},\quad n\geqslant0$$
which shows that $\mathscr{B}(f)$ is an entire function of exponential type $1/R$. Moreover from the integral representation of the Gamma function $\Gamma$, we have
$$\displaystyle n!=\int_0^\infty e^{-t}t^ndt$$
or for a fixed $z$, $|z|\less R$ and $n\geq0$ an integer
$$\displaystyle a_nz^n=\int_0^{+\infty}e^{-t}\dfrac{a_n}{n!}(zt)^ndt.$$
If $\epsilon\bg0$ is small enough so that $\kappa:=|z|(1/R+\epsilon)\less1$, then for $n\geq n_\epsilon$ and $t\bg0$ one has:
$$\displaystyle e^{-t}\sum_{n\geq n_\epsilon}\dfrac{|a_nz^n|t^n}{n!}\leqslant e^{-t}\sum_{n\geq n_\epsilon}\dfrac{(t\kappa)^n}{n!}\leqslant e^{-(1-\kappa)t}$$
which is integrable on $[0,+\infty[$. By dominated convergence 
$$\displaystyle \sum_{n\geq 0}a_nz^n=\sum_{n\geq 0}\int_0^\infty e^{-t}\dfrac{a_n}{n!}(zt)^n=\int_0^\infty e^{-t}\mathscr{B}(f)(tz)dt.$$
On the other hand if $\gamma$ is a positively oriented circle contained in $D(0,R)$, then 
$$\displaystyle\dfrac{1}{2i\pi}\int_\gamma f(\xi)e^{z/\zeta}\dfrac{d\zeta}{\zeta}=\dfrac{1}{2i\pi}\int_\gamma\sum_{n\geq 0}a_n\zeta^{n-1}e^{z/\zeta}d\zeta=\sum_{n,m\geq 0}a_n\dfrac{z^m}{m!}\dfrac{1}{2i\pi}\int_\gamma\zeta^{n-m-1}d\zeta=\mathscr{B}(f)(z).$$
It is of some interest to note that the Borel transformation is a particular case of a general picture. Let $\mathscr{H}(\C)$ be the space of entire functions and $\mathscr{H}^\prime(\C)$ its dual space. it is the space of of all functionals $L:\mathscr{H}(\C)\to\C$, $\C$-linear maps, bounded in the following sense: there exists two positive constants $C$, $R$ such that for all $f\in\mathscr{H}(\C)$
$$|L(f)|\leqslant C|f|_R$$
where $|f|_R=\sup\{|f(z)|=R\}=\sup\{|f(z)|\leqslant R\}$. The functional are called Analytic functionals. They are entirely determined by the moments $T(z^n),n\geq 0$. Actually for $T\in \mathscr{H}^\prime(\C)$ and $\displaystyle f(z)=\sum_{n\geq 0}a_nz^n\in\mathscr{H}(\C)$, the series $\displaystyle\sum_{n\geq 0}a_nT(z^n)$ converges absolutely and we have 
$$\displaystyle T(f)=\sum_{n\geq 0}a_nT(z^n).$$ 
A fundamental example of analytic functionals is given by entire functions of exponential type. Indeed if 
$$\displaystyle\varphi(u)=\sum_{n\geq 0}b_n\dfrac{u^n}{n!}$$
is such a function and if $f(z)=\displaystyle \sum_{n\geq 0}a_nz^n$ is any entire function, we may define the action of $\varphi$ on $f$ by:
$$\displaystyle T_\varphi(f)=\sum_{n\geq 0}a_nb_n$$
the series being absolutely convergent and $\varphi$ defines an analytic functional $T_\varphi\in \mathscr{H}^\prime(\C)$. If $\displaystyle f(z)=\sum_{n\geq 0}a_nz^n$ defines an analytic functional $T_f$, then 
$$\displaystyle T_f(z\mapsto e^{z\zeta})=\sum_{n\geq 1}\dfrac{a_n}{n!}\zeta^n=\mathscr{B}_f(\zeta)$$
and more generally if $T$ is an analytic functional, its Fourier-Borel transform is
$$\mathscr{F}_T(u)=T_z(z\mapsto e^{uz}).$$
It is an entire function of exponential type. Like the Fourier transform, the Fourier-Borel transform converts differentiation into multiplication by $\zeta$. In fact it is easily seen that 
\begin{equation}
\label{5506}
\mathscr{F}_{T\circ d/dz}(\zeta)=\zeta\mathscr{F}_T(\zeta)
\end{equation}
and
\begin{equation}
\label{5507}
\mathscr{F}_{T\circ z}(\zeta)=\dfrac{d}{d\zeta}\mathscr{F}_{T}(\zeta).
\end{equation}
In particular for any integers $s\geq 0$, $t\geq 0$ and for any $v\in\C$, the Fourier-Borel transform of the functional $T:f\mapsto \left(\dfrac{d}{dz}\right)^tz^sf_{|_{z=v}}$ is
$$\mathscr{F}_T(\zeta)=\left(\dfrac{d}{d\zeta}\right)^s(z^te^{vz}_{|_{z=\zeta}})$$
which can be seen directly from the equality 
$$\displaystyle\left(\dfrac{d}{dz}\right)^s(z^te^{vz})_{|_{z=\zeta}}=\sum_{k=0}^{\min\{s,t\}}\dfrac{s!t!}{k!(s-k)!(t-k)!}\zeta^{t-k}v^{s-k}e^{\zeta v}$$
and by the symmetry of the hand side under the transformation
$$(\zeta,v,s,t)\mapsto(v,\zeta,t,s).$$
\begin{remark}
By a theorem of Martineau the duals of some spaces of analytic functions are isomorphic under the Fourier-Borel transform to some other spaces of entire functions, a phenomenon akin to what happens in the case of the Paley-Wiener theorem. See \cite[Théorème 1, p.~121]{martineau1} 
\end{remark}
Let us write the Borel transform as the Laplace transform
$$\displaystyle f(z)=\int_0^\infty e^{-t}\mathscr{B}(f)(tz),\quad \dfrac{1}{z}f(1/z)=\int_0^\infty e^{-tz}\mathscr{B}(f)(t)dt.$$
Then by \eqref{5506} and \eqref{5507}:
$$\displaystyle\int_0^\infty e^{-xz}(\frac{d}{dz})^kz^mg(z)dz=x^m(-\frac{d}{dx})^m\frac{1}{x}f(1/x),\quad g=\mathscr{B}(f).$$
If we assume that $f(x)$ satisfies a Fuchsian differential equation, the function $\dfrac{1}{x}f(1/x)$ also satisfies a linear differential equation that can be put in the form
\begin{equation}
\label{5508}
\displaystyle\sum_{k=0}^K\sum_{m=0}^MA_{km}x^k(-\dfrac{d}{dx})^m\frac{1}{x}f(1/x)=0
\end{equation}
then by the integral relation for the Borel transform
$$\displaystyle \int_0^\infty e^{-xz}\sum_{k=0}^K\sum_{m=0}^MA_{km}(\dfrac{d}{dz})^kz^mg(z)dz=0.$$
Hence
\begin{equation}
\label{5600}
\displaystyle \sum_{k=0}^K\sum_{m=0}^MA_{km}(\dfrac{d}{dz})^kz^mf(z)dz=0.
\end{equation}
\begin{example}
\begin{itemize}
\item We take $g(z)=e^z=\displaystyle\sum_0^\infty \frac{z^n}{n!}$, $f(z)=\dfrac{1}{1-z}=\displaystyle\sum_{n\geq 0}z^n,|z|\less1$. The function $h(z):=\dfrac{1}{z}f(1/z)$ verifies $h-h^\prime +zh^\prime=0$ so that 
$$A_{00}=1, A_{01}=1, A_{10}=0, A_{11}=-1.$$
Hence $f$ must verify the equation $f-f^\prime=0$ as it must be.
\item A less trivial example is to consider the Bessel function $\displaystyle J_0(x)=\sum_0^\infty\frac{(-1)^r}{(r!)^2}(\frac{x}{2})^{2r}$ which satisfies
$$\displaystyle \left(x^2\frac{d^2}{dx^2}+x\frac{d}{dx}+x^2\right)J_0(x)=0.$$
The differential equation satisfied by $\dfrac{1}{x}J_0(\dfrac{1}{x})$ is:
$$x^4h^{\prime\prime}(x)+3x^3h^\prime(x)+(x^2+1)h(x)=0,\quad h(x)=\dfrac{1}{x}J_0(\dfrac{1}{x}).$$
The coefficients in \eqref{5507} are ($K=4,\,M=2$):
$$\begin{cases}
A_{00}&=1,\quad A_{10}=0,\quad A_{20}=1,\quad A_{30}=0,\quad A_{40}=0,\\
A_{01}&=0,\quad A_{11}=0,\quad A_{21}=0,\quad A_{31}=-3,\quad A_{41}=0,\\
A_{02}&=0,\quad A_{12}=0,\quad A_{22}=0,\quad A_{32}=0,\quad A_{42}=1.
\end{cases}$$
\end{itemize}
The Fourier-Borel transform of $J_0(x)$, that is $\displaystyle f(x)=\sum_{r=0}^\infty\dfrac{(-1)^r}{(r!)^3}(\dfrac{x}{2})^{2r}$ should satisfy the equation \eqref{5600}. This gives
$$x^2f^{\prime\prime\prime\prime}(x)+5xf^{\prime\prime\prime}+4f^{\prime\prime}(x)+xf(x)=0.$$
\end{example}

\section{The delta operator and the Polylogarithm function}
\noindent We have seen the relevance of the operator $\displaystyle \dfrac{d}{dx}\circ x   $. Its powers are linked to those of the delta operator $ \delta= x \dfrac{d}{dx}$ by
 \[\displaystyle \left(\dfrac{d}{dx} x\right)^{n+1}= \dfrac{d}{dx}\circ \left(x\dfrac{d}{dx} \right)^n\circ x.   \]  
 As in \cite{sebdjo}  it is possible  to make explicit the powers $\displaystyle \delta^n    $ as differential operators in $ \displaystyle \frac{d}{dx}  $, with monic coefficients, that is as elements of $ \BC[x][  \frac{d}{dx} ]   $. Actually we have:
\begin{proposition}
For every smooth function $f$ of the  variable $x$
\begin{eqnarray*}
  \aligned
   \delta^n f= &x\frac{df}{dx}+\Big[ \frac{1^{n-1}}{1-2}+  \frac{2^{n-1}}{2-1}       \Big]x^2\frac{{d^2f}}{dx^2}\\
  +&\Big[ \frac{1^{n-1}}{(1-2)(1-3)}+  \frac{2^{n-1}}{(2-1)(2-3)}+ \frac{3^{n-1}}{(3-1)(3-2)} \Big]x^3\frac{{d^3f}}{dx^3}\\
  + &\Big[\frac{1^{n-1}}{(1-2)(1-3)(1-4)}+  \frac{2^{n-1}}{(2-1)(2-3)(2-4)}+ \frac{3^{n-1}}{(3-1)(3-2)(3-4)}+\\
  +&  \frac{4^{n-1}}{(4-1)(4-2)(4-3)}  \Big]x^4\frac{{d^4f}}{dx^4}\\
  {}\\
  {}+&\cdots\\
  {}\\
  +&\Big[ \frac{1^{n-1}}{(1-2)(1-3)\cdots(1-n)}+  \frac{2^{n-1}}{(2-1)(2-3)\cdots(2-n)}+ \frac{3^{n-1}}{(3-1)(3-2)(3-4)\cdots(3-n)}\\
  {}&\quad \quad \quad +  \frac{n^{n-1}}{(n-1)(n-2)\cdots(n-(n-1))}    \Big]x^n\frac{{d^nf}}{dx^n}
  \endaligned
\end{eqnarray*}
 \end{proposition}

 This is a local problem and by  the Stone-Weierstrass theorem, we restrict ourselves to monic functions  $f(x)= x^k,\,k\in \BZ^+$. It is equivalent to show the following lemma:

\begin{lemma}\label{lem} For each real $u$:
\begin{eqnarray*} \label{hilbert}
  \aligned
   u^n= &u+\Big[ \frac{1^{n-1}}{1-2}+  \frac{2^{n-1}}{2-1}       \Big]u(u-1)\\
  +&\Big[ \frac{1^{n-1}}{(1-2)(1-3)}+  \frac{2^{n-1}}{(2-1)(2-3)}+ \frac{3^{n-1}}{(3-1)(3-2)} \Big]u(u-1)(u-2)\\
 + &\Big[\frac{1^{n-1}}{(1-2)(1-3)(1-4)}+  \frac{2^{n-1}}{(2-1)(2-3)(2-4)}+ \frac{3^{n-1}}{(3-1)(3-2)(3-4)}+\\
  +&  \frac{4^{n-1}}{(4-1)(4-2)(4-3)}  \Big]u(u-1)(u-2)(u-3)\\
  {}\\
  {}+&\cdots\\
  {}\\
  +&\Big[ \frac{1^{n-1}}{(1-2)(1-3)\cdots(1-n)}+  \frac{2^{n-1}}{(2-1)(2-3)\cdots(2-n)}+ \frac{3^{n-1}}{(3-1)(3-2)(3-4)\cdots(3-n)}\\
 \ {}&\quad \quad \quad +  \frac{n^{n-1}}{(n-1)(n-2)\cdots(n-(n-1))}    \Big]u(u-1)(u-2)\cdots(u-(n-1)).
  \endaligned
  \end{eqnarray*}
  \end{lemma}
\begin{proof}
 The Hilbert  polynomials $H_j(X)$ are extensions of binomial coefficients:
 \begin{equation}\label{Hilbert1}
 H_0(X)=1,\quad H_j(X)=\binom{X}{j}=
    \frac{X(X-1)\cdots (X-j+1)}{j!}.
  \end{equation}
 The family $ (H_j(X))_{0\leq j\leq n }  $ is an algebraic basis of the space $\BR_n[X]$ of polynomials of degrees at most  $n$. For each $n\in \BZ^+,$  we have
 \begin{equation}\label{Hilbert2}
  X^n= \sum_{j\geq 0} A_{nj} H_j(X)
 \end{equation}
with
\[A_{nj}= 0 \quad j>n.  \]
Several formulas for $A_{nj}$ are known, in particular \cite{sebdjo} 
\begin{equation}\label{Hilbert3}
 A_{nj}= \sum_{l_i>0,l_1+l_2+\cdots l_j=n} \frac{n!}{l_1!l_2!\cdots l_j!}
\end{equation}
or
\begin{equation}\label{Hilbert4}
 A_{nj}= \sum_{m=0}^j(-1)^m  \binom{j}{m}(j-m)^n= \Delta^j {X^n}_{|X=0}
\end{equation}
where $ \Delta $ is the difference operator $ \Delta f(x)= f(x+1)-f(x)    $. This ends the proof of the lemma.\end{proof}

It is possible to give another approach to the problem of the calculation of $\delta^nf=(x\dfrac{d}{dx})^nf    $ by integral representations, at least for the class of analytic functions of exponential
 type, in order to explain and to fix our ideas.

 \begin{theorem}In order that  $f(z)$ should be an integral function of order $1$ and type $ \gamma $
it is necessary and sufficient that $f(z)$ should be of the form
\begin{equation}\label{Borel1}
f(z)= \frac{1}{2i\pi}\int_{\mathcal C} e^{zu} \phi(u)\, du
\end{equation}
where ${\mathcal C}$ is a contour which includes the circle $ \displaystyle \{\vert u\vert = \gamma \}   $ and
 $ \phi $ is holomorphic in the open disc $   \displaystyle \{\vert u\vert < \gamma \}     $. In this case we have the inversion formula
 \begin{equation}\label{Borel2}
 \phi(u)= \int_0^{\infty} f(x)e^{xu}\,dx,\quad \gamma <u.
 \end{equation}
 \end{theorem}
 
 The formula \eqref{Borel1} gives, with $\displaystyle D_z= \frac{d}{dz}  $:
 \begin{equation}\label{Borel3}
 (zD_z)^nf(z)=  \frac{1}{2i\pi}\int_{\mathcal C} (zD_z)^n   e^{z u} \phi(u)\, du.
 \end{equation}
This integral representation has the advantage to reduce the problem to the computation of $   (zD_z)^n    e^{zu}$ which is a classical one. We introduce the exponential polynomials
 $\displaystyle \Phi_n    $ by
 \[   (xD_x)^n e^{x}= \Phi_n(x) e^{x}.\]
 They admit the generating series:
 \begin{equation}\label{Borel4}  e^{x(e^z-1)}= \sum_{n=0}^\infty \Phi_n(x) \frac{z^n}{n!}.       \end{equation}
 These polynomials satisfy many identities, through derivations of \eqref{Borel4} with respect to $x$ or to $z$.
 The equality \eqref{Borel3} becomes the following integral representation
 \[  (zD_z)^nf(z)=    \frac{1}{2i\pi}\int_{\mathcal C}  \Phi_n(zu)  e^{z u} \phi(u)\, du\]
which is a kind of a multiplicative convolution.
\begin{remark}
If we were interested by Cauchy type representation formula, we could use
\[(zD_z)^n  \frac{1}{u-z}= \sum _{k=1}^\infty \frac{k^n}{u^{k+1}}z^k.\]
A particular case is, for $n\in \BZ^+   $
\begin{equation}\label{poly1}
(zD_z)^n  \frac{z}{1-z}= \sum _{k=1}^\infty k^n z^k.
\end{equation}
The  inverses of the delta operator $\displaystyle \delta= x \frac{d}{dx}$ are sufficiently interesting. A possible one is given by
\[\delta^{-1}F(x)= \int_0^{x} F(u) \frac{du}{u}\]
acting on functions $F$ for which the integral converges. An easy induction shows that for $n\in \BZ^+   $
\begin{equation}\label{poly2}
(zD_z)^{-n}  \frac{z}{1-z}= \Phi(z,n)
\end{equation}
where
\[\Phi(z,s)= \sum_{k=1}^\infty \frac{z^k}{k^s}\]
is the classical polylogarithm function. The formulas \eqref{poly1} and \eqref{poly2} extended by
defining $ (zD_z)^0      $ as the identity operator give $\displaystyle (zD_z)^{n}  \frac{z}{1-z}$ for each $n\in \BZ   $.
\end{remark}

\section{The rotation group ${\rm ISO}(2)$}
Rotations of $\R^3$ about the $z$ axis form a $1$-parameter subgroup
$$R(\theta)=\left(\begin{array}{ccc}\cos(\theta) & -\sin(\theta) & 0 \\\sin(\theta) & \cos(\theta) & 0 \\0 & 0 & 1\end{array}\right)=\exp(\theta M_3)$$
for some $M_3\in \mathfrak{so}(3)$, with 
$$M_3=\dfrac{d}{d\theta}\exp(\theta M_3)_{|_{\theta=0}}=\dfrac{d}{d\theta} R(\theta)_{|_{\theta=0}}=\left(\begin{array}{ccc}0 & -1 & 0 \\1 & 0 & 0 \\0 & 0 & 0\end{array}\right).$$
A basis for $\mathfrak{so}(3)$ is given by 
$$M_1=\left(\begin{array}{ccc}0 & 0 & 0 \\0 & 0 & -1 \\0 & 1 & 0\end{array}\right),\quad M_2=\left(\begin{array}{ccc}0 & 0 & 1 \\0 & 0 & 0 \\-1 & 0 & 0\end{array}\right),\quad M_3=\left(\begin{array}{ccc}0 & -1 & 0 \\1 & 0 & 0 \\0 & 0 & 0\end{array}\right)$$
and we have for the Lie algebra structure 
$$\left[M_1,M_2\right]=E_3,\quad \left[M_2,M_3\right]=M_1,\quad \left[M_3,M_1\right]=M_2$$
giving the structure constants for $\mathfrak{so}(3)$. A general motion on the euclidian plane is given by
\begin{equation*}
\begin{split}
 x^\prime&=x\cos(\theta)-y\sin{\theta}+a\\
 x^\prime&=x\sin(\theta)+y\cos{\theta}+b
 \end{split}
 \end{equation*}
 $a,b\in\R,0\leqslant\theta\less 2\pi$. We consider
 \begin{equation*}
 \begin{split}
 \R^2\rtimes S^1=\R^2&\rtimes[0,2\pi[\longrightarrow SO(3)\\
 & g=(a,b;\theta)\mapsto T(g)=\left(\begin{array}{ccc}\cos(\theta) & -\sin(\theta) & a \\\sin(\theta) & \cos(\theta) & b \\0 & 0 & 1\end{array}\right).
 \end{split}
 \end{equation*}
 The Lie algebra of the group can be obtained by calculating, at the identity, the derivative of the matrix $T(g)$ with respect to its three parameters. We obtain
 $$A=\left(\begin{array}{ccc}0 & 0 & 1\\0 & 0 & 0 \\0 & 0 & 0\end{array}\right),\quad B=\left(\begin{array}{ccc}0 & 0 & 0\\0 & 0 & 1 \\0 & 0 & 0\end{array}\right),\quad C=\left(\begin{array}{ccc}0 & 1 & 0\\-1 & 0 & 0 \\0 & 0 & 0\end{array}\right).$$
 The infinitesimal generators in the space of functions defined on the plane are 
 $$P_x=\dfrac{\partial}{\partial x},\quad P_y=\dfrac{\partial}{\partial y},\quad L_z=x\dfrac{\partial}{\partial y}-y\dfrac{\partial}{\partial x}.$$
 To clarify this point we need some computations for $\exp(-tA)$, $\exp(-tB)$ and $\exp(-tC)$.
 
 We have $A^2=0$ then $\exp(-tA)=I-tA=\left(\begin{array}{ccc}1 & 0 & -t\\0 & 1 & 0 \\0 & 0 & 1\end{array}\right)$. This gives 
 $$f\left(e^{-tA}\left(\begin{array}{c}x \\y \\1\end{array}\right)\right)=f(x-t,y,z)$$
 and
 $$-\dfrac{d}{dt}f\left(e^{-tA}\left(\begin{array}{c}x \\y \\1\end{array}\right)\right)_{|_{t=0}}=\dfrac{\partial}{\partial x}f(x,y,z).$$
 \begin{equation*}
 \begin{split}
 \dfrac{d}{dt}f\left(e^{-tC}\left(\begin{array}{c}x \\y \\1\end{array}\right)\right)_{|_{t=0}}&=\dfrac{d}{dt}_{|_{t=0}}f(x\cos(t)-y\sin(t),x\sin(t)+y\cos(t),1)\\
 &=(x\dfrac{\partial}{\partial y}-y\dfrac{\partial}{\partial x})f.
 \end{split}
 \end{equation*}
  $$B=\left(\begin{array}{ccc}0 & 0 & 0 \\0 & 0 & 1 \\0 & 0 & 0\end{array}\right),\quad B^2=0.$$
 $$\exp(-tB)=I-tB=\left(\begin{array}{ccc}1& 0 & 0 \\0 & 1 & -t \\0 & 0 & 1\end{array}\right)$$
 $$-\dfrac{d}{dt}f\left(e^{-tB}\left(\begin{array}{c}x \\y \\1\end{array}\right)\right)_{|_{t=0}}=-\dfrac{d}{dt}_{|_{t=0}}f(x,y-t,z)=\dfrac{\partial}{\partial y}f(x,y,z).$$
 In the dimension $n=3$, the antisymmetric matrices are connected to the cross product
 $$\left(\begin{array}{c}a \\b \\c\end{array}\right)\wedge\left(\begin{array}{c}x \\y \\z\end{array}\right)=\left(\begin{array}{ccc}0& -c& b \\c & 0 & -a \\-b & a & 0\end{array}\right)\left(\begin{array}{c}x \\y \\z\end{array}\right)$$
 if $\Omega=\left(\begin{array}{c}a \\b \\c\end{array}\right)$ and $A_\Omega=\left(\begin{array}{ccc}0& -c& b \\c & 0 & -a \\-b & a & 0\end{array}\right)$, we see that $A_\Omega$ is the general form of antisymmetric matrices in the case of dimension $n=3$. The matrix $\exp(A_\Omega)$ is the Rotation of axis $\R\Omega$ and of angle $\displaystyle \|\Omega\|\sqrt{a^2+b^2+c^2}$. To see this choose a rotation $U$ transforming $\displaystyle \Omega$ into $\displaystyle \|\Omega\|e_1=(\|\Omega\|,0,0)$. The rotation $U$ leaves the cross product invariant
 $$\displaystyle UA_\Omega U^{-1}=A_{\|\Omega\|e_1}=\left(\begin{array}{ccc}0& 0& 0\\0 & 0 & -\|\Omega\| \\0 & \|\Omega\| & 0\end{array}\right)$$
 so that
 $$\exp(A_{\|\Omega\|e_1})=\left(\begin{array}{ccc}1& 0& 0\\0 & \cos\|\Omega\| & -\sin\|\Omega\| \\0 & \sin\|\Omega\| & \cos(\|\Omega\|\end{array}\right)$$
 which is the Rotation around around $\R e_1$ and angle $\|\Omega\|$. From
 $$\exp(A_\Omega)=U^{-1}\exp(A_{\|\Omega\| e_1})U$$
 we obtain the desired result. 
 
 We apply this remark to the matrix $C=\left(\begin{array}{ccc}0 & 1 & 0\\-1 & 0 & 0 \\0 & 0 & 0\end{array}\right)$. A rotation of axis $\R e_2$ and of angle $\theta$ is
 $$R_\theta=\left(\begin{array}{ccc}\cos(\theta) & 0 & \sin(\theta)\\0 & 0 & 1 \\-\sin(\theta) & 0 & \cos(\theta)\end{array}\right)$$
  and here $\Omega=\left(\begin{array}{c}0 \\0 \\-1\end{array}\right)=-e_3$. If $\theta=-\pi/2$, $R_{-\pi/2}=\left(\begin{array}{ccc}0 & 0 & -1\\0 & 1 & 0 \\1 & 0 & 0\end{array}\right)=U$ transforms $-e_3$ into $e_1$. The inverse matrix is $R_{\pi/2}=\left(\begin{array}{ccc}0 & 0 & 1\\0 & 1 & 0 \\-1 & 0 & 0\end{array}\right)$. In conclusion 
  \[\|\Omega\|e_1=e_1,\quad tA_{e_1}=\left(\begin{array}{ccc}0 & 0 & 0\\0 & 0 & -t \\0 & t & 0\end{array}\right),\]
  and
   \[\exp(-tA_{e_1})=\left(\begin{array}{ccc}1 & 0 & 0\\0 & \cos(t) & \sin(t) \\0 & -\sin(t) & \cos(t)\end{array}\right),\quad
  \exp(-tC)=\left(\begin{array}{ccc}\cos(t) & -\sin(t) & 0\\\sin(t) & \cos(t) & 0 \\0 & 0 & 1\end{array}\right).
  \]

 We have the following commutation relations:
 \begin{equation*}
 \begin{split}
 \left[L_z,P_x\right]&=-P_y,\qquad\left[C,A\right]=-B\\
 \left[L_z,P_y\right]&=P_x,\qquad\left[C,B\right]=A\\ 
  \left[P_x,P_y\right]&=0,\qquad\left[A,B\right]=0.
 \end{split}
 \end{equation*}
 We introduce
 the polar coordinates $x=r\cos(\varphi),\,y=r\sin(\varphi)$ and the operators
 $$L_z^1=\dfrac{1}{i}\dfrac{\partial}{\partial\varphi},P_{\pm}=e^{\pm i\varphi}(\pm\dfrac{\partial}{\partial r}+\dfrac{i}{r}\dfrac{\partial}{\partial\varphi}).$$
We suppose that $P_{\pm}$ act as Ladder operators on a basis $(\varphi_n)_{n\in\N}$ of a separable Hilbert space $V$. That is if $|n>$ is a basis function of $V$, then
 $$P_+|n>=-|n+1>,\quad P_-|n>=|n-1>, |n>\leftrightarrow\varphi_n.$$
 Let us also choose the basis functions $|n>$ to be eigenfunctions of $L_z$ with eigenvalue $n$: $L_z|n>=n|n>$. The equation
 $$-i\dfrac{\partial<r,\varphi|n>}{\partial\varphi}=n<r,\varphi |n>,\quad<r,\varphi |n>\leftrightarrow\varphi_n(r,\varphi)$$
 implies that $<r,\varphi |n>=J_n\left(r\right)e^{in\varphi}$ for some function $J_n\left(r\right)$. In the discrete case
 $$<r,\varphi|P_+P_- |n>:=P_+P_-(\varphi_n(r,\varphi))=-<r,\varphi|P_+ |n-1>=<r,\varphi|n>.$$  
  So that $<r,\varphi|n>$ is an eigenfunction of $P_+P_-$ with an eigenvalue equal to $1$. For the continuous case, we compute on a basis $|r,\varphi>$
  \begin{equation*}
  \begin{split}
  P_+P_-<r,\varphi|n>&=P_+e^{-i\varphi}(-\dfrac{\partial}{\partial r}+\dfrac{i}{r}\dfrac{\partial}{\partial\varphi})J_n\left(r\right)e^{in\varphi}\\
  &=(-J_n^{\prime\prime}\left(r\right)-\dfrac{1}{r}J^\prime_n\left(r\right)+\dfrac{n^2}{r^2}J_n\left(r\right))e^{in\varphi}=J_n\left(r\right)e^{in\varphi}.
  \end{split}
  \end{equation*}
  Thus these requirements imply that $J_n\left(r\right)$ satisfies the Bessel differential equation. Also the two relations
  \begin{equation}\label{5509}P_+<r,\varphi|n>=-<r,\varphi|n+1>\end{equation}
  \begin{equation}\label{5510}P_-<r,\varphi|n>=-<r,\varphi|n-1>\end{equation}
  with $P_\pm=e^{\pm i\varphi}(\pm\dfrac{\partial}{\partial r}+\dfrac{i}{r}\dfrac{\partial}{\partial \varphi})$ and $<r,\varphi|n>=J_n\left(r\right)e^{in\varphi}$. The relation \eqref{5509} give the fundamental recurrence relation for the Bessel functions. Indeed
  $$e^{i\varphi}(\dfrac{\partial}{\partial r}+\dfrac{i}{r}\dfrac{\partial}{\partial \varphi})J_n\left(r\right)e^{in\varphi}=-J_{n+1}\left(r\right)e^{i(n+1)\varphi}$$
 or 
 \begin{equation}
 \label{5511}
 J_n^\prime\left(r\right)-\dfrac{n}{r}J_n\left(r\right)=J_{n+1}\left(r\right).\end{equation}
 Similarly using \eqref{5510} instead of \eqref{5509} we get
 \begin{equation}
 \label{5512}
 J_n^\prime\left(r\right)-\dfrac{n}{r}J_n\left(r\right)=J_{n-1}\left(r\right).\end{equation}
 If we add the identities \eqref{5511} and \eqref{5512} we obtain the classical relation
 $$\dfrac{2n}{r}J_n\left(r\right)=J_{n-1}\left(r\right)+J_{n+1}\left(r\right).$$
 Finally if we take their difference we get
  $$2J_n^\prime\left(r\right)=J_{n-1}\left(r\right)-J_{n+1}\left(r\right).$$
  \section{From Frobenius determinants to non-euclidian geometry}
  A particular example of Frobenius determinant is the determinant of a circulant matrices. Let $K$ be a field and $Circ_n$ the algebra of circulant matrices of order $n$ with coefficients in $K$, then as algebras we have the isomorphisms
  $$Circ_n\simeq K^n\simeq K[X]/(X^n-1).$$
  If $K=\Q$ and we write
  $$X^n-1=\displaystyle\prod_{d|n}\phi_d(X),$$
   where $\phi_m(X)$ is the associated $m$-th cyclotomic polynomial. The M\"obius function $\mu(n)$ is defined by the conditions: $\mu(1)= 1,\;\mu(n)= 0$ if $n$ is divisible by the square of a prime and $\mu(n)= (-1)^r  $ if $n=p_1p_2\cdots p_r$ for disctinct primes $p_i$. The $m$-th cyclotomic polynomial is given by
  \[\phi_m(X)= \prod_{d\vert m}(X^d-1)^{\mu(m/d)}.\]
  We then get the identification
$$=\displaystyle\bigoplus_{d|n}\Q(\zeta_d)$$
as étale algebras and where $\Q(\zeta_d)$ is the cyclotomic extension corresponding to the $d$-th primitive roots of unity. In case $K=\C$, $Circ_n$ is also isomorphic as étale algebras to
$$\C\oplus\ldots\oplus\C, $$
$n$ times.

 To give an explicit example of the link of $Circ_n$ with geometry, we fix $n=2$. The general case is developed in \cite{ssvv}. In this case we have two determinants
 $$\left|\begin{array}{cc}x & iy \\iy & x\end{array}\right|=x^2+y^2$$
 and
 $$\left|\begin{array}{cc}x & y \\y & x\end{array}\right|=x^2-y^2$$
 to which we associate
 \begin{itemize}
 \item The Laplace equation, the euclidian metric $ds^2=dx^2+dy^2$ and the circles $\{x^2+y^2=r^2\}$.
 \item The wave equation, the pseudo-metric $ds^2=dx^2-dy^2$ and the hyperbolas $\{x^2-y^2=r^2\}$.
 \end{itemize}
 To these metrics we attach (resp.) the distance functions: If $M_1=(x_1,y_1)$, $M_2=(x_2,y_2)$, we have
 $$d_C(M_1,M_2)=\displaystyle{(x_1-x_2)^2+(y_1-y_2)^2}$$
 respectively
 $$d_H(M_1,M_2)=\displaystyle{(x_1-x_2)^2-(y_1-y_2)^2}.$$
 $d_H$ is clearly not an euclidian metric but it has a very rich group theoretical properties. The pseudo-rotation group is the group of transformations which leaves the quadratic form in $\R^n$
 $$Q(x)=x_1^2+\ldots+x_p^2-x_{p+1}^2-\ldots-x_{n}^2,\,0\less p\less n$$
 invariant. This group is named the Lorenz group when $n=4$ and $p=3$ and in general it is denoted by $O_{p,n-p}$. The simplest situation is when $n=2$ and $p=1$ so that the group $O_{1,1}$ leaves the quadratic form $x^2-y^2$ invariant. There is a distinguished one parameter subgroup of $O_{1,1}$, namely $\{t\mapsto h(t)\}$ with 
 $$h(t)=\left(\begin{array}{cc}\cosh(t) & \sinh(t) \\\sinh(t) & \cosh(t)\end{array}\right).$$
 We denote it $SO^+_{1,1}$. It is the component of the identity in $O_{1,1}$. Moreover $O_{1,1}$ contains also $\eta=\left(\begin{array}{cc}1 & 0 \\0 & -1\end{array}\right)$
 and we have the decomposition in connected components
 $$O_{1,1}=SO^+_{1,1}\bigcup\eta SO^+_{1,1}.$$  
 A generator of that one parameter group is
 $$L=x\dfrac{\partial}{\partial y}+y\dfrac{\partial}{\partial x}.$$
 If we make the change of variable $y=iv$, then $x^2-y^2=x^2+v^2$ and any transformation leaving $x^2-y^2$ invariant leaves $x^2+v^2$ also invariant.. The group of the quadratic form $x^2+v^2$ is $SO(2)$ with elements given by 
 $$\begin{cases}
x^\prime&=x\cos(\theta)+v\sin(\theta)\\
y^\prime&=-x\sin(\theta)+v\cos(\theta)
\end{cases}$$
and if $\alpha=i\theta$ we obtain
$$\begin{cases}
x^\prime&=x\cosh(\alpha)+y\sinh(\alpha)\\
y^\prime&=x\sinh(\alpha)+y\cosh(\alpha)
\end{cases} $$
so that $SO^+_{1,1}$ may be seen as a group of rotations through imaginary angles.

Macdonald functions $K_\nu$ are modified Bessel functions satisfying the differential equation
$$z^2u^{\prime\prime}+zu^\prime-(z^2+\nu^2)u=0.$$
They admit the following integral representation
$$\displaystyle K_\nu(z)=\dfrac{\pi}{\Gamma(1/2+\nu)}(z/2)^{\nu}\int_1^\infty e^{-tz}(t^2-1)^{\nu-1/2}dt,\,|\arg(z)|\less\frac{\pi}{2}, \Re\nu\bg -1/2.$$
Useful properties that we will discuss in large are 
$$-2K_\nu^\prime(z)=K_{\nu-1}(z)+K_{\nu+1}(z)$$
$$-2\nu/zK_\nu(z)=K_{\nu-1}(z)-K_{\nu+1}(z).$$
For $\nu=n$ an integer let us give a Lie theoretical interpretation of the Macdonald functions as we did in the previous section.

Pseudo-Rotations of $\R^3$ about the $z$ axis form a $1$-parameter subgroup
$$R(\varphi)=\left(\begin{array}{ccc}\cosh(\varphi) & \sinh(\varphi) & 0 \\\sinh(\varphi) & \cosh(\varphi) & 0 \\0 & 0 & 1\end{array}\right).$$
A general motion the pseudo-euclidean plane is
\begin{equation*}
\begin{split}
 x^\prime&=x\cosh(\varphi)+y\sinh{\varphi}+a\\
 x^\prime&=x\sinh(\varphi)+y\cosh{\varphi}+b
 \end{split}
 \end{equation*}
 $a,b\in\R,0\leqslant\theta\less 2\pi$. It is the action of the semi-direct product
 \begin{equation*}
 \R^2\rtimes H^1
 \end{equation*}
where $H^1$ is the hyperbola $\{x^2-y^2=1\}$, on the plane with coordinates $^t(x,y,1)$. Every element of  $ \R^2\rtimes H^1$
 can be written $g=(a,b;\varphi)\mapsto S(g)=\left(\begin{array}{ccc}\cosh(\varphi) & -\sinh(\varphi) & a \\\sinh(\varphi) & \cosh(\varphi) & b \\0 & 0 & 1\end{array}\right).$

 As Lie algebra of the group, we obtain
 $$a_1=\left(\begin{array}{ccc}0 & 0 & 1\\0 & 0 & 0 \\0 & 0 & 0\end{array}\right),\quad a_2=\left(\begin{array}{ccc}0 & 0 & 0\\0 & 0 & 1 \\0 & 0 & 0\end{array}\right),\quad a_3=\left(\begin{array}{ccc}0 & 1 & 0\\1 & 0 & 0 \\0 & 0 & 0\end{array}\right).$$
 The infinitesimal generators in the space of functions defined on the plane are 
 $$P_x^1=\dfrac{\partial}{\partial x},\quad P_y^1=\dfrac{\partial}{\partial y},\quad L_z^2=x\dfrac{\partial}{\partial y}+y\dfrac{\partial}{\partial x}.$$

 The commutation relations below hold:
 \begin{equation*}
 \begin{split}
 \left[L_z^2,P_x^1\right]&=P_y^1,\qquad\left[a_3,a_1\right]=a_2\\
 \left[L_z^2,P_y^1\right]&=P_x^1,\qquad\left[a_3,a_2\right]=a_1\\ 
  \left[P_x^1,P_y^1\right]&=0,\qquad\left[a_1,a_2\right]=0.
 \end{split}
 \end{equation*}
 Let us consider the polar coordinates $x=r\cosh(\varphi),\,y=r\sinh(\varphi)$ corresponding to the hyperbola. We introduce
 $$L_z^3=\dfrac{\partial}{\partial\varphi},P_{\pm}^1=e^{\pm \varphi}(\pm\dfrac{\partial}{\partial r}-\dfrac{1}{r}\dfrac{\partial}{\partial\varphi}).$$
 We suppose that $P_{\pm}^1$ act as Ladder operators on a basis $(\varphi_n)_{n\in\N}$ of a separable Hilbert space $V$. That is if $|n>$ is a basis function of $V$, then
 $$P_+^1|n>=-|n+1>,\quad P_-^1|n>=-|n-1>, |n>\leftrightarrow\varphi_n.$$
 Moreover choose the basis functions $|n>$ to be eigenfunctions of $L_z^4$ with eigenvalue $-n$: $L_z|n>=-n|n>$. The equation
 $$\dfrac{\partial<r,\varphi|n>}{\partial\varphi}=-n<r,\varphi |n>,\quad<r,\varphi |n>\leftrightarrow\varphi_n(r,\varphi)$$
 implies that $<r,\varphi |n>=J_n\left(r\right)e^{-n\varphi}$ for some function $K_n\left(r\right)$. In the discrete case
 $$<r,\varphi|P_+^1P_-^1 |n>:=P_+^1P_-^1(\varphi_n(r,\varphi))=-<r,\varphi|n>.$$  
  So that $<r,\varphi|n>$ is an eigenfunction of $P_+^1P_-^1$ with an eigenvalue equal to $-1$. For the continuous case, we compute on a basis $|r,\varphi>$
  \begin{equation*}
  \begin{split}
  P_+^1P_-^1<r,\varphi|n>&=P_+^1e^{-\varphi}(-\dfrac{\partial}{\partial r}-\dfrac{1}{r}\dfrac{\partial}{\partial\varphi})K_n\left(r\right)e^{-n\varphi}\\
 &=e^{\varphi}(\dfrac{\partial}{\partial r}-\dfrac{1}{r}\dfrac{\partial}{\partial\varphi})e^{-\varphi}(-\dfrac{\partial}{\partial r}-\dfrac{1}{r}\dfrac{\partial}{\partial\varphi})K_n\left(r\right)e^{-n\varphi}\\
  &=(-K_n^{\prime\prime}\left(r\right)-\dfrac{1}{r}K_n^\prime\left(r\right)+\dfrac{n^2}{r^2}) e^{-n\varphi}\\&=-K_n\left(r\right)e^{-n\varphi}.
  \end{split}
  \end{equation*}
  Thus these requirements imply that $K_n\left(r\right)$ satisfies the Bessel-Macdonald differential equation. Also the two relations below hold
  \begin{equation}\label{5609}P_+^1<r,\varphi|n>=-<r,\varphi|n-1>\end{equation}
  \begin{equation}\label{5610}P_-^1<r,\varphi|n>=<r,\varphi|n+1>\end{equation}
  with $P_\pm^1=e^{\pm\varphi}(\pm\dfrac{\partial}{\partial r}-\dfrac{1}{r}\dfrac{\partial}{\partial \varphi})$ and $<r,\varphi|n>=K_n\left(r\right)e^{-n\varphi}$. The relation \eqref{5609} gives
  $$e^{\varphi}(\dfrac{\partial}{\partial r}-\dfrac{1}{r}\dfrac{\partial}{\partial \varphi})K_n\left(r\right)e^{-n\varphi}=-K_{n-1}\left(r\right)e^{-(n-1)\varphi}$$
 or 
 \begin{equation}
 \label{5611}
 K_n^\prime\left(r\right)+\dfrac{n}{r}K_n\left(r\right)=-K_{n-1}\left(r\right).\end{equation}
 Similarly using \eqref{5610} instead of \eqref{5609} we get
 \begin{equation}
 \label{5612}
 -K_n^\prime\left(r\right)+\dfrac{n}{r}K_n\left(r\right)=K_{n+1}\left(r\right).\end{equation}
 If we add the identities \eqref{5611} and \eqref{5612} we obtain the classical relation
 $$\dfrac{2n}{r}K_n\left(r\right)=-K_{n-1}\left(r\right)+K_{n+1}\left(r\right).$$
 Finally if we make the difference between them we get
  $$2K_n^\prime\left(r\right)=-K_{n-1}\left(r\right)-K_{n+1}\left(r\right).$$


\begin{thebibliography}{SKK2}
\bibitem{bateman1}
H.~Bateman. The solution of partial differential equations by means of definite integrals.
\emph{Proc. Lond. Math. Soc. (2)}, 1:451--458, 1904.
\bibitem{beukers2001}
F. Beukers. Differential Galois theory, chap. 8 of "From Number theory to Physics", Springer-Verlag, 1992.
 \emph{Springer-Verlag}, 2001.
\bibitem{chaundy}
T.~ W.~Chaundy. A method for solving certain linear partial differential equations.
\emph{Proc. Lond. Math. Soc.}, 21(s2):214--234, 1921.
\bibitem{conrad}
K.~Conrad.  On the origin of representation theory. Enseign. Math. 44, 361-392, 1998.
\bibitem{dolgachev}
I.~V.~Dolgachev. Lectures on modular forms. Available at:\\
\url{www.math.lsa.umich.edu/~idolga/modular.pdf}
.\bibitem{donaldson}
S.~Donaldson. Riemann surfaces. Oxford GTM, 22. OUP, Oxford, 2011
\bibitem{gelfand2000}
I.~M.~Gelfand, S.~G.~Gindikin, M.~I.~Graev. Selected topics in integral geometry. Trans. Math. Mono., 220,  2000.
\bibitem{shilov}
I.~M.~Gelfand, G.E~Shilov. Generalized Functions, Vol.1. Ac. Press, New York, 1965.
 \bibitem{hormander}
L.~H\"ormander. The analysis of linear partial differential operators. I, Grundlehren der mathematischen wissenschaften, vol. 256, Springer-Verlag, Berlin, 1983.
\bibitem{john1938}
F.~John. The ultrahyperbolic differential equation with four independent variables  
\emph{Duke Math. J.}, 4:300--322, 1938.
\bibitem{kampe}
J.~Kamp\'e de F\'eriet. La fonction hyperg\'eom\'etrique, M\'emorial des Sciences math\'ematiques, Fascicule 85,1937. Available at:\\
\url{www.numdam.org/item?id=MSM_1937_85_1_0 }
\bibitem{kashiwara}
M.~Kashiwara,  T.~ Kawai, T.~Kimura. Foundations of algebraic analysis, Princeton Mathematical Series, vol. 37, PUP, Princeton, NJ, 1986.
\bibitem{martineau1}
A.~Martineau. Equations différentielles d'ordre infini.
\emph{Bull. SMF}, 95:109--1154, 1967
\bibitem{sebdjo}
 D.~Meguedmi and A.~Sebbar. Fa\`a di Bruno's Formula and Modular Forms. \emph{Complex Analysis and Operator Theory, Volume 10 (2)-Sep 28, 2015.}
 \bibitem{miller}
 W.~ Miller Jr., Symmetry Groups and Their Applications. New York: Acad. Press, 1972.
\bibitem{penrose1967}
R.~Penrose. Twistor algebra.
 \emph{J. Math. Phys.}, 8:345--366, 1967.
\bibitem{radon1917}
J.~Radon. Über die Bestimmung von Funktionen durch ihre Integral werte längs gewisser Mannigfaltigkeiten
\emph{Sächs. Akad. Wiss.} Leipzig, Maths. Nat. Kl 69:262--277, 1917.
\bibitem{ssvv}
A.~Sebbar, D.C.~Struppa, A.~Vajiac and M.~Vajiac. Book in preparation.
\bibitem{sing2003}
M.~F. Singer and M.~Van Der Put. Galois theory of Linear Differential Equations.
\emph{Grundlehren der mathematischen wissenschaften,} 328. Springer, 2003.
\bibitem{singerman}
G.~A.~Jones and D.~Singerman. Complex Functions-an algebraic and geometric viewpoint, CUP, 1987.
\bibitem{talman}
J.~D.~Talman. Special Functions: A Group Theoretic Approach. W. A. Benjamin, Inc., New York-Amsterdam, 1968. Based on lectures by Eugene P. Wigner. With an introduction by Eugene P. Wigner.
\bibitem{vilenkin}
N.Ya.~Vilenkin and A.U.~Klimyk. Special functions, group representations, and integral transforms, Vol. 1, Kluwer, 1991.
\bibitem{WW}
E.T.~Whittaker and G.~N. Watson. A course of modern analysis, 4th ed. CUP, 1935.                                                                                                                                                                                                                                                                                                                                                                                                                                                                                                                                                                            

\end{thebibliography}
\end{document}